\documentclass[11pt,halfparskip]{scrartcl}

\usepackage{a4wide}
	\usepackage[latin1]{inputenc}
	\usepackage{setspace,amssymb,amsthm,dsfont,graphicx,enumerate}
\usepackage[numbers]{natbib}

	\usepackage{amsmath}	

	\usepackage{caption}
	\usepackage{color}

	\usepackage{float}
	\restylefloat{table}

	\usepackage[paper=a4paper,left=30mm,right=30mm,top=30mm,bottom=35mm]{geometry}
	\usepackage{hyperref}
	\usepackage{enumitem} 
\usepackage{caption}
\usepackage{subcaption}

\def\Bka{{\it Biometrika}}

\def \R {\mathbb{R}}
\def \N {\mathbb{N}}

\def \cS {\mathcal S}

\theoremstyle{plain}
\newtheorem{theorem}{Theorem}

\newtheorem{corollary}[theorem]{Corollary}

\newcounter{definitionAlgo}
\theoremstyle{definition}
\newtheorem{example}{Example}
\newtheorem{assumption}{}

\newcounter{definitionCount}
\begin{document}
\LARGE
\vspace*{0.5cm}
\begin{center}
{\bfseries Nonparametric identification and maximum likelihood estimation for hidden Markov models\\*[0.4cm]}
\large
\textbf{Grigory Alexandrovich, Hajo Holzmann\footnote{Address for correspondence: Prof.~Dr.~Hajo Holzmann, 
Philipps-Universität Marburg,
Fachbereich Mathematik und Informatik, 
Hans-Meerweinstr. 
D-35032 Marburg, Germany
email: holzmann@mathematik.uni-marburg.de,
Fon: + 49 6421 2825454
} and Anna Leister}\\*[0.4cm]
{\sl Fakult\"at f\"ur Mathematik und Informatik, Philipps-Universit\"at Marburg, Germany}
\end{center}
\normalsize
\begin{abstract}
Nonparametric identification and maximum likelihood estimation for finite-state hidden Markov models are investigated. We obtain identification of the parameters as  well as the order of the Markov chain if the transition probability matrices have full-rank and are ergodic, and if
the state-dependent distributions are all distinct, but not necessarily linearly independent. 
Based on this identification result, we develop nonparametric maximum likelihood estimation theory. First, we show that the asymptotic contrast, the Kullback--Leibler divergence of the hidden Markov model, identifies the true parameter vector nonparametrically as well. Second, for classes of state-dependent densities which are arbitrary mixtures of a parametric family, we show consistency of the nonparametric maximum likelihood estimator. Here, identification of the mixing distributions need not be assumed. 
Numerical properties of the estimates as well as of nonparametric goodness of fit tests are investigated in a simulation study.  
\end{abstract}


\section{Introduction}
A discrete-time hidden Markov model consists of an observed process $(Y_t)_{t \in \N}$ as well as a latent, unobserved process $(X_t)_{t \in \N}$, such that 
the $Y_t$ are independent given the $X_t$, the conditional distribution of $Y_s$ given the $X_t$ depends on $X_s$ only and $X_t$ is a finite-state Markov chain. We assume that $X_t$ is time-homogeneous. 
The cardinality $K$ of the state space of $X_t$ is called the number of states. The conditional distributions of $Y_s$ given $X_s = k$ ($k=1, \ldots, K$) are called the state-dependent distributions, and  we assume that they are independent of $s$. The entries of the transition probability matrix are denoted by $\Gamma= (\alpha_{j,k})_{j,k=1, \ldots, K}$. Further, we assume that the $Y_t$ take values in any subset of Euclidean space $\cS \subset \R^q$, and denote the distribution functions of the state-dependent distributions by $F_k$ ($k=1, \ldots, K$). 


Parametric estimation theory for finite-state hidden Markov models is well-developed; see \citet{ler} for consistency and \citet{bickel} for asymptotic normality of the maximum likelihood estimator. 
%
%
In order to achieve greater flexibility and to avoid misspecification, nonparametric modelling and estimation of the component distributions have received recent interest, see \citet{HolzmannLeister}, \citet{gassiat} and \citet{Vernet2015}. 
However, the most basic question is whether such models are identifiable. We give an affirmative answer in great generality: if the transition probability matrix $\Gamma$ is ergodic and of full rank, and if the state-dependent distributions are all distinct, then the parameters, together with the number of states, are all identifiable. 
Our second main result states that the asymptotic contrast for maximum likelihood  estimation, the generalized Kullback--Leibler divergence of the hidden Markov model, uniquely identifies the true parameter nonparametrically. It is well known that the ordinary Kullback--Leibler divergence discriminates between any two probability distributions on the same measurable space without reference to a particular model. However, for hidden Markov models, for which the generalized Kullback--Leibler divergence is defined as a limit of normalized log-likelihoods, the contrast property had previously been deduced from mere parametric identification of mixtures of product distributions in \citet{ler}, and had thus been restricted to parametric settings. Our second result allows us to investigate consistency of the maximum likelihood estimator over nonparametric classes. As an important example, we consider general mixtures of a parametric family as model for the state-dependent distributions, and as a third main result obtain consistency of the mixture densities under suitable assumptions. Here, we do not assume that the mixing distributions themselves are identified, and thus allow, e.g., general mixtures of normals.     

Let us discuss how our identification results relate to previous ones in the literature. 
In a seminal paper, based on a result by \citet{kruskal} on the identification of factors in three-way tables, \citet{allman} showed generic identifiability of various latent-state models, including hidden Markov models with finite-valued observations. 
Strict point identification, up to label swapping, for general-valued hidden Markov models was recently discussed by \citet{gassiat} and \citet{gassiat2}. 
Using analytic arguments, \citet{gassiat2} showed that if $\Gamma$ has full rank, and if the state-dependent distributions are from a location family of an arbitrary density, then all parameters as well as the number of components are identified from the joint distribution of two observations. While certainly of interest, merely the assumption of equal scale in each component which is implied by the model may be too restrictive for most applications.  
For a given $K$, \citet{gassiat} show identification if $\Gamma$ has full-rank and if the state-dependent distributions are linearly independent. The result follows by combining arguments given in \citet{allman} for generic identification of hidden Markov models and of finite mixtures of product distributions.
While the assumption of  linearly independent state-dependent distributions is convenient in the proofs, it is not intuitive, and also difficult to interpret for nonparametric classes such as smooth classes of densities, or shape-constrained classes such as log-concave densities, where more than two distinct distributions may well be linearly dependent. 
Our result for distinct state-dependent distributions is better suited for such nonparametric classes. In its proof, the main challenge is to find a substitute for the linear independence of the state-dependent distributions. 


\section{Nonparametric identification}\label{sec:results}

Our basic assumptions are as follows. 
\begin{assumption}\label{eq:tpm}
The transition probability matrix $\Gamma= (\alpha_{j,k})_{j,k=1, \ldots, K}$ of $(X_t)$ has full rank and is ergodic.%
\end{assumption}
\begin{assumption}\label{eq:statedep}
The state-dependent distributions $F_k$ ($k=1, \ldots, K$) are all distinct.
\end{assumption}
Let us first consider the stationary case for a fixed number of components. 
\begin{assumption}\label{eq:stationary}
The Markov chain $(X_t)$ is stationary with starting distribution $\pi$, the stationary distribution of $\Gamma$. 
\end{assumption}
\begin{theorem}\label{cor:identagainstgeneralparstat}
For given $K$, let $\Gamma$, $F_1, \ldots, F_K$ and $\tilde \Gamma$, $\tilde F_1, \ldots, \tilde F_K$ be two sets of parameters for a $K$-state hidden Markov model, such that the joint distributions of $\big(Y_1, \ldots, Y_{2 K +1}\big)$ under both sets of parameters are equal. Further, suppose that $\Gamma$ and $F_1, \ldots, F_K$ satisfy Assumptions \ref{eq:tpm}--\ref{eq:stationary}. Then both sets of parameters coincide up to label swapping. 
\end{theorem}

In Theorem \ref{cor:identagainstgeneralparstat}, Assumptions \ref{eq:tpm} and \ref{eq:statedep} solely concern $\Gamma$, $F_1, \ldots, F_K$; nothing is assumed for $\tilde \Gamma$, $\tilde F_1, \ldots, \tilde F_K$. For statistical inference, this implies that Assumptions \ref{eq:tpm}--\ref{eq:stationary} are required for the true model, but estimators need not be restricted to satisfy these constraints. 
\begin{example}
To show the necessity of the full-rank assumption of the transition probability matrix, we construct for each $K \geq 1$ a $(K+1)$-state matrix of rank $K$ and two sets of $K+1$ distributions, which are even linearly independent, such that the observations in a resulting $(K+1)$-state hidden Markov model have the same distribution. To this end, let $\Gamma= (\alpha_{j,k})_{j,k=1, \ldots, K}$ be a $K$-state ergodic transition probability matrix of full rank. Let $\delta, \beta \in (0,1)$ with $\delta \not= \beta$, set $ p = \beta /(1+\beta - \delta)$ for which $p \in (0,1)$, and consider the $(K+1)$-state matrix of rank $K$
\[ \Gamma_1 = \begin{pmatrix} \alpha_{1,1} & \cdots & \alpha_{1,K-1} & p\,\alpha_{1,K} & (1-p)\,\alpha_{1,K}\\ 
\vdots &  & \vdots &  & \vdots\\
\alpha_{K-1,1} & \cdots & \alpha_{K-1,K-1} & p\,\alpha_{K-1,K} & (1-p)\,\alpha_{K-1,K}\\
\alpha_{K,1} & \cdots & \alpha_{K,K-1} & p\,\alpha_{K,K} & (1-p)\,\alpha_{K,K}\\
\alpha_{K,1} & \cdots & \alpha_{K,K-1} & p\,\alpha_{K,K} & (1-p)\,\alpha_{K,K}
\end{pmatrix}.\]
Let $F_1, \ldots, F_{K+1}$ be linearly independent distribution functions, for example, normal distributions with distinct parameters. As the second set $\tilde F_1, \ldots, \tilde F_{K+1}$ of distribution functions let $ \tilde F_1  = F_1, \ldots, \tilde F_{K-1} = F_{K-1}$ and 
\begin{align*}
\tilde F_{K} & = \delta \, F_K + (1-\delta)\, F_{K+1}, \quad \tilde F_{K+1} = \beta \, F_K + (1-\beta)\, F_{K+1}.
\end{align*}
Then
$ p \, \tilde F_{K} +(1-p) \, \tilde F_{K+1} =  p \,  F_{K} +(1-p) \, F_{K+1}$, 
and from \citet{HolzmannSchwaiger2013}, the distributions of the observations of a $(K+1)$-state hidden Markov model with transition probability matrix $\Gamma_1$, stationary starting distribution and either set of state-dependent distributions are equal to that of a $K$-state stationary hidden Markov model with transition probability matrix $\Gamma$ and state dependent distributions $F_1, \ldots, F_{K-1}$ and $p \,  F_{K} +(1-p) \, F_{K+1}$.  
\end{example}
\begin{example}
One may wonder whether the assumption of distinct state-dependent distributions is actually necessary for identification, or whether states may possibly be reconstructed merely from transitions if there are sufficiently many different state-dependent distributions, without all of them being distinct. 
In this example we describe a class of hidden Markov models where this is not possible. A stationary Markov chain $(X_t)_{t \in \N}$ with transition probability matrix $\Gamma$ is called lumpable with respect to a partition $\{ G_1, \ldots, G_m\}$ of the state-space $\{1, \ldots, K\}$ if the process $(\tilde X_t)_{t \in \N}$ defined by $\tilde X_t = j$ if $X_t \in G_j$ ($j=1, \ldots, m$) is also a Markov chain. \citet{Kemeny} show that this is equivalent to $\text{pr}\,(X_{t+1} \in G_j \mid X_t \in G_i) = \text{pr}\,(X_{t+1} \in G_j \mid X_t =k)$ ($i,j=1, \ldots, m$; $k \in G_i$). If this is the case in a hidden Markov model $(Y_t, X_t)_{t \in \N}$, for which the state-dependent distributions are equal over the states in the elements $G_j$ ($j=1, \ldots, m$) of the partition, then its distribution reduces to that of a $m$-state hidden Markov model. In particular, the full transition probability matrix $\Gamma$ of the $K$-state representation cannot be identified. 
\end{example}
\begin{example}
Hidden Markov models with state-dependent densities which mainly differ in terms of their scale are used for modelling financial time series. The distinct scales correspond to volatility states of the market;  
see \citet{HolzmannSchwaiger2013} and references therein. 
In case of three states, there is a transition regime between highest and lowest volatility. Hence it is plausible that the state-dependent density with intermediate scale may actually be a mixture of the two densities with lowest and highest volatility, thus making the three state-dependent densities linearly dependent. See the supplementary material for simulations in such a scenario. 
\end{example}

Now let us turn to the case of a general starting distribution. While only of moderate statistical interest by itself, identification of the initial distribution is an essential tool for proving the nonparametric contrast property of the Kullback--Leibler divergence of a hidden Markov model in Section \ref{sec:mlcontrast}. The choice of $T$ in the following theorem is due to the fact that for $t_0 = K^2-2K+2$, $\Gamma^{t_0}$ has strictly positive entries \citep{holladay}.
\begin{theorem}\label{the:identgeneral}
For a known number of states $K$, let $\lambda, \Gamma$, $F_1, \ldots, F_K$ and $\tilde \lambda, \tilde \Gamma$, $\tilde F_1, \ldots, \tilde F_K$ be two sets of parameters for a $K$-state hidden Markov model, where $\lambda$ and $\tilde \lambda$ denote the initial distributions of the Markov chain. Suppose that the joint distributions of $ \big(Y_1, \ldots, Y_{T}\big)$
with $T = (2 K+1) (K^2-2K+2) + 1$, are equal under both sets of parameters. Further, suppose that $\Gamma$ and $F_1, \ldots, F_K$ satisfy Assumptions \ref{eq:tpm} and \ref{eq:statedep}. Then both sets of parameters coincide up to label swapping. 
\end{theorem}

Finally, let us turn to the additional identification of the number of states. 
For $L<K$ we may interpret an $L$-state as a $K$-state hidden Markov model, where $K-L$ states are never visited by the underlying Markov chain. From Theorem \ref{the:identgeneral}, we therefore get the following corollary.  
\begin{corollary}\label{the:identorder}
Let $\lambda, \Gamma$ and $F_1, \ldots, F_K$  and $\bar \lambda, \bar \Gamma$ and $\bar F_1, \ldots, \bar F_L$ be two sets of parameters for a $K$-state and a $L$-state hidden Markov model, where $L \leq K$. Assume that $\Gamma$ is ergodic and of full rank, and that $F_1, \ldots, F_K$ are all distinct. If the joint distributions of $ \big(Y_1, \ldots, Y_{T}\big)$, $T = (2 K+1) (K^2-2K+2) + 1$, are the same under the both sets of parameters, then $K=L$ and the sets of parameters are equal up to a label swapping.
\end{corollary}

In summary, we get the following identification result for the number of states and the parameters. 

\begin{corollary}
For a hidden Markov model, within the class of parameters satisfying Assumptions \ref{eq:tpm} and \ref{eq:statedep}, both the number of states and the parameters are identified from the distribution of the observed process $(Y_t)_{t \in \N}$. 
\end{corollary}
Indeed, if we compare two hidden Markov models with $L$ and $K$ states satisfying Assumptions \ref{eq:tpm} and \ref{eq:statedep} and having equal distributions of the observations, then Theorem \ref{the:identgeneral} takes care of the case $L=K$ while Corollary \ref{the:identorder} shows that the case $L \not= K$ cannot occur. 

\section{Nonparametric maximum likelihood estimation}\label{sec:nonparmle}
\subsection{The Kullback--Leibler divergence of a hidden Markov model}\label{sec:mlcontrast}
Let $\mathcal D$ be a class of densities on $\cS$ with respect to some $\sigma$-finite measure $\nu$. 
Suppose that $(Y_t, X_t)_{t \in \N}$ is a $K$-state hidden Markov model with transition probability matrix $\Gamma_0$ satisfying Assumptions \ref{eq:tpm} and \ref{eq:stationary} and having stationary distribution $\pi_0$, and that the state-dependent distributions $F_{1,0}, \ldots, F_{K,0}$ are all distinct and have densities $f_{1,0}, \ldots ,f_{K,0}$  from the class $\mathcal D$. In the following, we write $Y_s^t = (Y_s, \ldots, Y_t)$ and $y_s^t = (y_s, \ldots, y_t)$ ($1 \leq s < t < \infty$).

For parameters $\lambda$, $\Gamma$, $f_1, \ldots, f_K$, $n \in \N$ and $y_1^n \in \cS^n$ consider 
\[ g_n\big(y_1^n; \lambda, \Gamma, f_1, \ldots , f_K \big) = \sum \limits_{x_1 = 1}^{K} \cdots \sum \limits_{x_n = 1}^{K}  \lambda_{x_1} f_{x_1}(y_1) \prod \limits_{i=2}^n \alpha_{x_{i-1}, \; x_i} f_{x_i}(y_i),\]
the joint density of $n$ observations under these parameters, and denote the log-likelihood function of $Y_1, \ldots, Y_n$ by
\[ L_{n}\big(\lambda, \Gamma, f_1, \ldots, f_K \big) = \log g_n\big(Y_{1}^{n}; \lambda, \Gamma, f_1, \ldots, f_K \big).\]
\begin{assumption}\label{eq:intble} 
The true densities $f_{j,0} \in \mathcal{D}$ satisfy $E \{|\log f_{j,0}(Y_1)|\} < \infty$,\quad  ($j=1,\hdots,K$). 
\end{assumption}

\begin{assumption}\label{eq:intpospart} 
The model satisfies $E \{\log f (Y_1) \}^+ < \infty$, \quad ($f \in \mathcal{D}$). 
\end{assumption}

\begin{theorem}\label{the:mlcontrast}
Suppose that $(Y_t, X_t)_{t \in \N}$ is a $K$-state hidden Markov model with transition probability matrix $\Gamma_0$ satisfying Assumptions \ref{eq:tpm} and \ref{eq:stationary}, and that the state-dependent distributions $F_{1,0}, \ldots, F_{K,0}$ are all distinct and have densities $f_{1,0}, \ldots ,f_{K,0}$  from the class $\mathcal D$, and satisfy Assumption \ref{eq:intble}.  Let $\lambda, \lambda_0$ be $K$-state probability vectors with strictly positive entries. 
Under Assumption \ref{eq:intpospart}, given $f_1, \ldots, f_K \in \mathcal D$ we have almost surely as $n \to \infty$ that
\begin{align}
\begin{split}\label{eq:kullbackhmm}
 n^{-1}\, \big\{L_{n}\big(\lambda, \Gamma, f_1, \ldots, f_K \big) & - L_{n}\big(\lambda_0,\Gamma_0, f_{1,0}, \ldots, f_{K,0}  \big) \big\}\\
	& \to - K\{(\Gamma_0, f_{1,0}, \ldots, f_{K,0}),(\Gamma, f_1, \ldots, f_K)\} \in (- \infty, 0],
\end{split}
\end{align}
and $K\{(\Gamma_0, f_{1,0}, \ldots, f_{K,0}),(\Gamma, f_1, \ldots, f_K) \} = 0$ if and only if the two sets of parameters are equal up to label swapping. 
\end{theorem}

The limit in (\ref{eq:kullbackhmm}) defines the Kullback--Leibler divergence of the hidden Markov model.  
As could be expected, it does not identify the initial distribution, and arbitrary probability vectors with positive entries, not necessarily the stationary distribution of $\Gamma$, can be used in the likelihood function. 
It is well known that the ordinary Kullback--Leibler divergence discriminates between any two probability distributions on the same measurable space without reference to a particular model. For hidden Markov models, \citet{ler} showed that the limit in (\ref{eq:kullbackhmm}) may be represented as an integral over the ordinary Kullback--Leibler divergence of finite segments of hidden Markov models, where integration is with respect to the initial distributions. From this and parametric identification of finite mixtures of product distributions, he deduced the contrast property, that is $K\{(\Gamma_0, f_{1,0}, \ldots, f_{K,0}),(\Gamma, f_1, \ldots, f_K)\} \in [0,  \infty)$ and $K\{(\Gamma_0, f_{1,0}, \ldots, f_{K,0}),(\Gamma, f_1, \ldots, f_K) \} = 0$ if and only if the two sets of parameters are equal up to label swapping, within a parametric class. However, Theorem \ref{the:identgeneral} implies that it holds without reference to a parametric family. 

%
%
%
%

\subsection{Nonparametric maximum likelihood estimation for state-dependent mixtures}\label{sec:nonparmlmix}

In this subsection we use arbitrary mixtures of some parametric family of state-dependent densities to illustrate how the above results can be employed for nonparametric maximum likelihood estimation in hidden Markov models. 
Suppose that $(f_{\vartheta})_{\vartheta\in \Theta}$ is a parametric family of densities on $\cS$ with respect to some $\sigma$-finite measure, and that $\Theta \subset \R^d$ is compact. Let $\tilde{\Theta}$ be the set of Borel probability measures on $\Theta$. Endowed with the weak topology $\tilde{\Theta}$ is also a compact set. Assume that the map $(y,\vartheta) \mapsto f_{\vartheta}(y)$ is continuous on $\cS\times \Theta$. Given $\mu \in \tilde{\Theta}$, we let 
\[ f_\mu(y) = \int_\Theta f_{\vartheta}(y)\, d \mu(\vartheta)\]
denote the corresponding mixture density. We shall call $\mu$ the mixing distribution for $f_\mu$, and take $\mathcal D = \{f_\mu:\, \mu \in \tilde \Theta\}$ as model for the state-dependent densities.  

For independent identically distributed observations there is some literature on nonparametric estimation of $\mu$ or $f_\mu$. \citet{lindsay} shows that there exists a nonparametric maximum likelihood estimator for $\mu$ with finite support, the number of support points being at most equal to the sample size. \citet{Leroux1992b} obtains its consistency under the assumption that the mixing distribution $\mu$ is identified from $f_\mu$. While convergence of estimators of $\mu$ may be quite slow \citep{ryden2005}, $f_\mu$ is estimated at optimal near-parametric rates for normal mixtures \citep{kim14, ghosal2001}.
We shall focus on consistency and in contrast to \citet{Leroux1992b} do not assume that 
the mixing distribution $\mu$ is identified from the mixture density $f_\mu$, since our interest is in the estimation of $f_\mu$ rather than of $\mu$. Thus we allow, e.g., arbitrary mixtures of normal densities in both mean and variance, for which the mixing distribution is not identified \citep{teicher}.  

Let 
$ \theta = \big(\Gamma, \mu_1, \ldots, \mu_K) \in G \times \tilde{\Theta} \times  \cdots \times \tilde{\Theta}$, 
where $G$ is the compact set of $K$-state transition probability matrices. Given the sequence of observations $Y_1, \ldots, Y_n$, the log-likelihood function is 
\[ L_n(\theta) = \log\, \big\{ \sum \limits_{x_1 = 1}^{K} \cdots \sum \limits_{x_n = 1}^{K}  \lambda_{x_1} f_{\mu_{x_1}}(Y_1) \prod \limits_{i=2}^n \alpha_{x_{i-1}, \; x_i} f_{\mu_{x_i}}(Y_i)\, \big\},\]
where $\lambda$ is an arbitrary $K$-state strictly positive probability vector. First, we show the existence of a maximum likelihood estimator for which the state-dependent mixing distributions $\mu_k$ ($k=1, \ldots, K$) have finite support. 
\begin{theorem}\label{prop:existencemle}
Let $(f_{\vartheta})_{\vartheta\in \Theta}$ be a parametric family of densities with $\Theta \subset \R^d$ compact, and let $\mathcal D = \{f_\mu:\, \mu \in \tilde \Theta\}$ be the model for the state-dependent densities, where $\tilde{\Theta}$ is the set of Borel probability measures on $\Theta$.
Then for any $n \geq 1$ there exists a maximum likelihood estimator $\hat \theta_n = (\hat \Gamma_n,\hat \mu_{1,n},\ldots,\hat \mu_{K,n})$, for which the state-dependent mixing distributions are of the form
\[ \hat \mu_{k,n} = \sum_{j=1}^m a_j\, \delta_{\vartheta_{j,k}} \quad (k=1, \ldots, K),\]
where $m\in \{1,\hdots,Kn+1\}$, $a_j>0$, $\sum_{j=1}^m a_j=1$, 
$\vartheta_{j,k}\in \Theta$ ($j=1, \ldots, m$), and where $\delta_\vartheta$ is the point-mass at $\vartheta$. 
\end{theorem}
This is similar to \citet{lindsay}'s existence result, although more components are required due to the distinct states and the non-convexity of the likelihood of the hidden Markov model. 

Let us turn to consistency. Assume that the true state-dependent densities $f_{k,0} = f_{\mu_{k,0}}$ belong to the model and are all distinct, and that $\Gamma_0$ satisfies Assumption \ref{eq:tpm}. 
\begin{assumption} \label{ass_finiteexp}
For every $\mu \in \tilde{\Theta}$ and a small enough neighborhood $O_{\mu}$ of $\mu$ we have
\begin{equation*}
E \big[\sup_{\tilde \mu\in O_\mu}\{\log f_{\tilde \mu}(Y_1)\}^{+}\big] < \infty.
\end{equation*}
\end{assumption}
%
%
\begin{theorem}\label{th:consistency}
Let $(f_{\vartheta})_{\vartheta\in \Theta}$ be a parametric family of densities with $\Theta \subset \R^d$ compact, and let $\mathcal D = \{f_\mu:\, \mu \in \tilde \Theta\}$ be the model for the state-dependent densities, where $\tilde{\Theta}$ is the set of Borel probability measures on $\Theta$.
Suppose that Assumptions \ref{eq:tpm}, \ref{eq:stationary}, \ref{eq:intble} and \ref{ass_finiteexp} hold, and let  $\hat \theta_n = (\hat \Gamma_n,\hat \mu_{1,n},\ldots,\hat \mu_{K,n})$ denote a maximum likelihood estimator. Then, after relabeling, we have in probability as $n \to \infty$ that $\hat \Gamma_n \to \Gamma_0$  and 
\[ f_{\hat \mu_{k,n}}(y) \to f_{k,0}(y) \quad (y \in \cS, \  k=1,\hdots, K).\]
Furthermore, if the mixing distribution $\mu$ is identified from the mixture density $f_\mu$, then we additionally have that $d_w\big(\hat \mu_{k,n},\mu_{k,0}\big) \to 0$ in probability, where $d_w$ is a distance which metrizes weak convergence in $\tilde \Theta$.  
\end{theorem}
\section{Simulations}\label{sec:sims}
\subsection{Nonparametric maximum likelihood estimation}\label{sec:simestimaton}
In this section we investigate the performance of the nonparametric maximum likelihood estimator based on state-dependent %
mixtures in a simulation study, and discuss its applications to goodness of fit assessment of parametric models. 

Consider a three-state hidden Markov model, in which the state-dependent densities are mixtures of univariate Gaussian distributions, specified as follows.
Let $g_{\beta(a,b)}(x)= \{\Gamma(a+b)\}/\{\Gamma(a)\Gamma(b)\}x^{a-1}(1-x)^{b-1}\mathds{1}_{(0,1)}(x)$ denote the density of the Beta distribution, $g_{\beta(a,b)}(x; l, s) = g_{\beta(a,b)}\{(x-l)/s\}/s$ the Beta density translated by $l$ and scaled by $s$ and let $\phi_{\mu,\sigma}$ denote the Gaussian density with parameters $\mu$ and $\sigma$. 
The state-dependent density of the first state is taken as $f_{1,0}(y)=0$$\cdot$$33 \phi_{-10,2}(y)+ 0$$\cdot$$33 \phi_{-7\cdot5,2}(y) + 0$$\cdot$$34\phi_{-4,2}(y)$. The densities of the second and third states, $f_{2,0}(y)$ and $f_{3,0}(y)$, are general mixtures of univariate Gaussian densities. For $f_{2,0}(y)$ we let $\mu$ follow the Beta distribution $g_{\beta(2,2)}(\mu;0,1)$ and let $\sigma$ be uniformly distributed on the interval $(1,4)$, for $f_{3,0}(y)$ we let $\mu$ follow $g_{\beta(2,11)}(\mu;5,33)$ and let $\sigma$ be uniformly distributed on $(1$$\cdot$$4,1$$\cdot$$6)$. 
For the transition probability matrix we choose
\begin{equation*}
\Gamma=
\begin{pmatrix}
\text{0$\cdot$5} & \text{0$\cdot$25} & \text{0$\cdot$25}\\
\text{0$\cdot$4} & \text{0$\cdot$4} & \text{0$\cdot$2}\\
\text{0$\cdot$2} & \text{0$\cdot$2} & \text{0$\cdot$6}
\end{pmatrix}.
\end{equation*}
Computing the nonparametric maximum likelihood estimator of a mixing distribution and the resulting mixture is not an easy task, see, e.g.,~\citet{laird1978}. We successively compute the maximum likelihood estimator for a given number of mixture components in each state using the expectation-maximization algorithm as described in \citet{Volant2013}, and increase the number of components as long as the resulting likelihood increases.  

In our simulations, we use series of lengths  $n=1000$ from the above model. We also consider the maximum likelihood estimators in two misspecified parametric hidden Markov models with simple Gaussian and two-component mixtures of Gaussian distributions, respectively.  
The nonparametric maximum likelihood estimators are denoted by $f_{\hat{\mu}_{k,n}}$, the simple Gaussian estimators by $f_{\tilde{\mu}_{k,n}}$ and the two-component Gaussian mixture estimators by $f_{\bar{\mu}_{k,n}}$ $(k=1,2,3)$.
On a computer with 3.07 GHz and 24GB RAM, computing the simple Gaussian estimators once requires $2.3$ seconds, the two-component Gaussian mixture estimator requires $8.7$ seconds and the nonparametric maximum likelihood estimators requires $84.5$ seconds. 

To illustrate the consistency of $f_{\hat{\mu}_{k,n}}$ as stated in Theorem \ref{th:consistency}, we evaluate the relative errors over 10000 simulations for the points indicated in Fig.~\ref{fig:estimates} and listed in Table \ref{tab:relerr}. The results together with those for the misspecified parametric estimators are given in Table \ref{tab:relerr}. The relative errors for $f_{\tilde{\mu}_{k,n}}$ and  $f_{\bar{\mu}_{k,n}}$ are higher at most points than those for $f_{\hat{\mu}_{k,n}}$, in particular for states $1$ and $3$, which reflects the bias of these estimators due to misspecification. The estimators for the transition probability matrices perform rather similarly for the three methods, therefore we do not report the results.
Additional simulation results for series of lengths different from 1000, which illustrate the consistency of the nonparametric maximum likelihood estimator and its performance for shorter series, are provided in the supplementary material. 
\begin{table}[t]
\centering
\begin{tabular}{rrrrrrrrrr}
 $y$& $-$15$\cdot$45 & $-$13$\cdot$77 & $-$11$\cdot$22 & $-$9$\cdot$05 & $-$7$\cdot$26 & $-$5$\cdot$3 & $-$2$\cdot$86 & $-$0$\cdot$21 & 1$\cdot$56 \\ 
nonpar & 109$\cdot$79 & 28$\cdot$00 & 6$\cdot$92 & 12$\cdot$94 & 23$\cdot$93 & 5$\cdot$09 & 43$\cdot$82 & 46$\cdot$40 & 43$\cdot$87 \\ 
  2-comp & 117$\cdot$75 & 28$\cdot$61 & 6$\cdot$26 & 12$\cdot$46 & 25$\cdot$18 & 4$\cdot$94 & 45$\cdot$04 & 48$\cdot$01 & 49$\cdot$49 \\ 
 Gauss & 136$\cdot$66 & 31$\cdot$14 & 5$\cdot$84 & 10$\cdot$68 & 24$\cdot$37 & 4$\cdot$68 & 43$\cdot$15 & 52$\cdot$93 & 37$\cdot$43 \\
   \\
  $y$  & $-$9$\cdot$36 & $-$6$\cdot$36 & $-$2$\cdot$71 & $-$0$\cdot$68 & 0$\cdot$5 & 1$\cdot$67 & 3$\cdot$71 & 7$\cdot$36 & 10$\cdot$36 \\ 
nonpar & 65$\cdot$27 & 22$\cdot$20 & 64$\cdot$95 & 9$\cdot$77 & 13$\cdot$44 & 19$\cdot$36 & 25$\cdot$00 & 59$\cdot$64 & 67$\cdot$53 \\ 
 2-comp & 69$\cdot$44 & 22$\cdot$63 & 68$\cdot$76 & 10$\cdot$60 & 13$\cdot$88 & 19$\cdot$48 & 25$\cdot$12 & 61$\cdot$55 & 67$\cdot$06 \\ 
 Gauss & 79$\cdot$61 & 16$\cdot$69 & 74$\cdot$73 & 9$\cdot$60 & 15$\cdot$02 & 19$\cdot$97 & 25$\cdot$32 & 81$\cdot$74 & 98$\cdot$08 \\ 
\\
 $y$   & 2$\cdot$27 & 3$\cdot$74 & 6 & 7$\cdot$99 & 9$\cdot$66 & 11$\cdot$61 & 14$\cdot$93 & 20$\cdot$17 & 22 \\ 
nonpar & 1090$\cdot$32 & 166$\cdot$99 & 9$\cdot$90 & 20$\cdot$26 & 13$\cdot$87 & 6$\cdot$38 & 7$\cdot$04 & 33$\cdot$61 & 48$\cdot$31 \\ 
2-comp & 1103$\cdot$22 & 175$\cdot$93 & 8$\cdot$29 & 22$\cdot$56 & 15$\cdot$08 & 5$\cdot$95 & 6$\cdot$81 & 37$\cdot$69 & 50$\cdot$26 \\ 
Gauss & 1236$\cdot$47 & 202$\cdot$98 & 4$\cdot$79 & 24$\cdot$17 & 18$\cdot$80 & 6$\cdot$69 & 3$\cdot$24 & 34$\cdot$78 & 52$\cdot$51 
\end{tabular}
\caption{Relative errors ($\times100$) of the three estimators compared to the true densities at selected values for $y$ averaged over 10000 replications. `Gauss' stands for Gaussian state-dependent distributions, `2-comp' for two component Gaussian mixtures and `nonpar' for nonparametric Gaussian mixtures. }
\label{tab:relerr}
\end{table}

\begin{figure}[ht]
\centering
\includegraphics[width=0.3\textwidth]{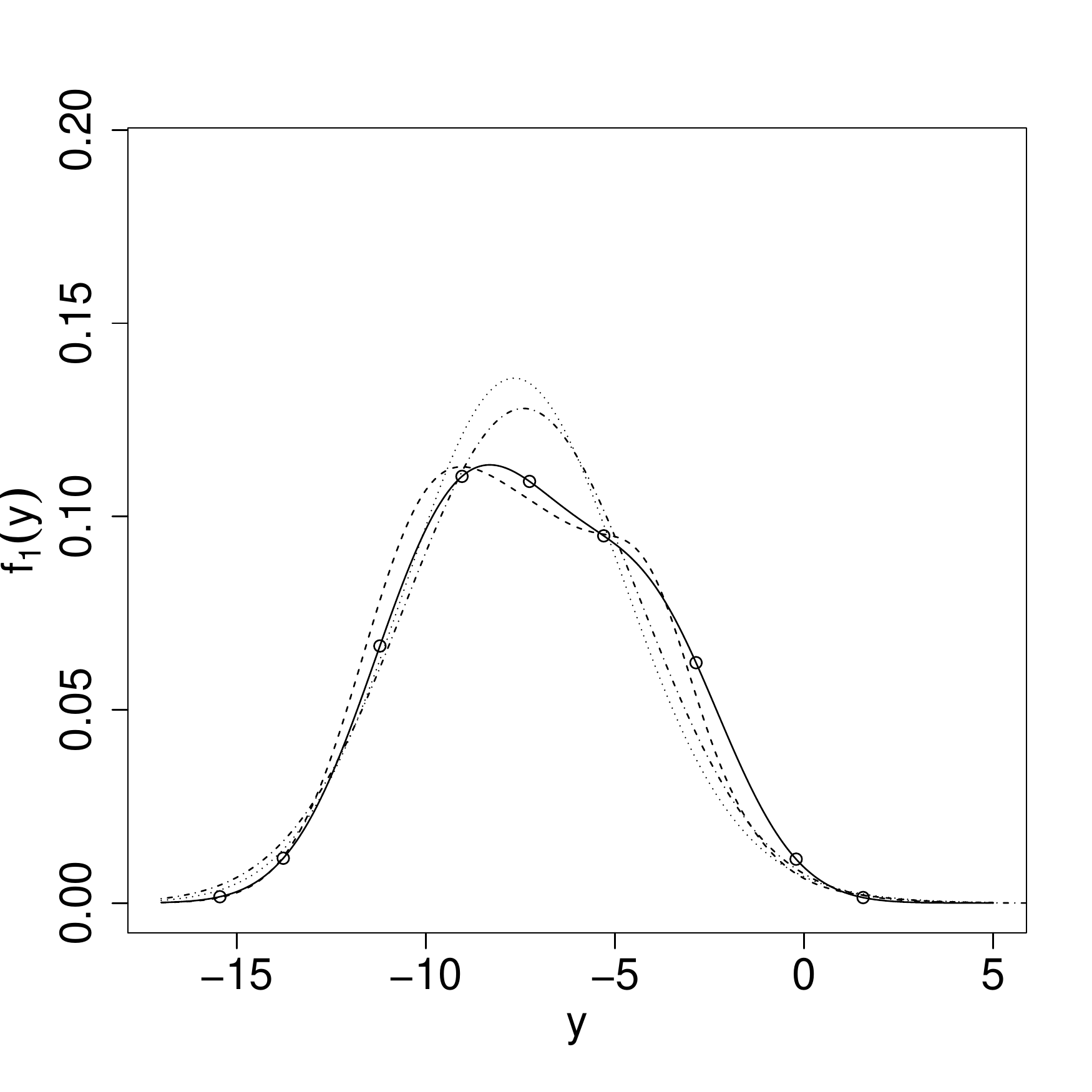}
\includegraphics[width=0.3\textwidth]{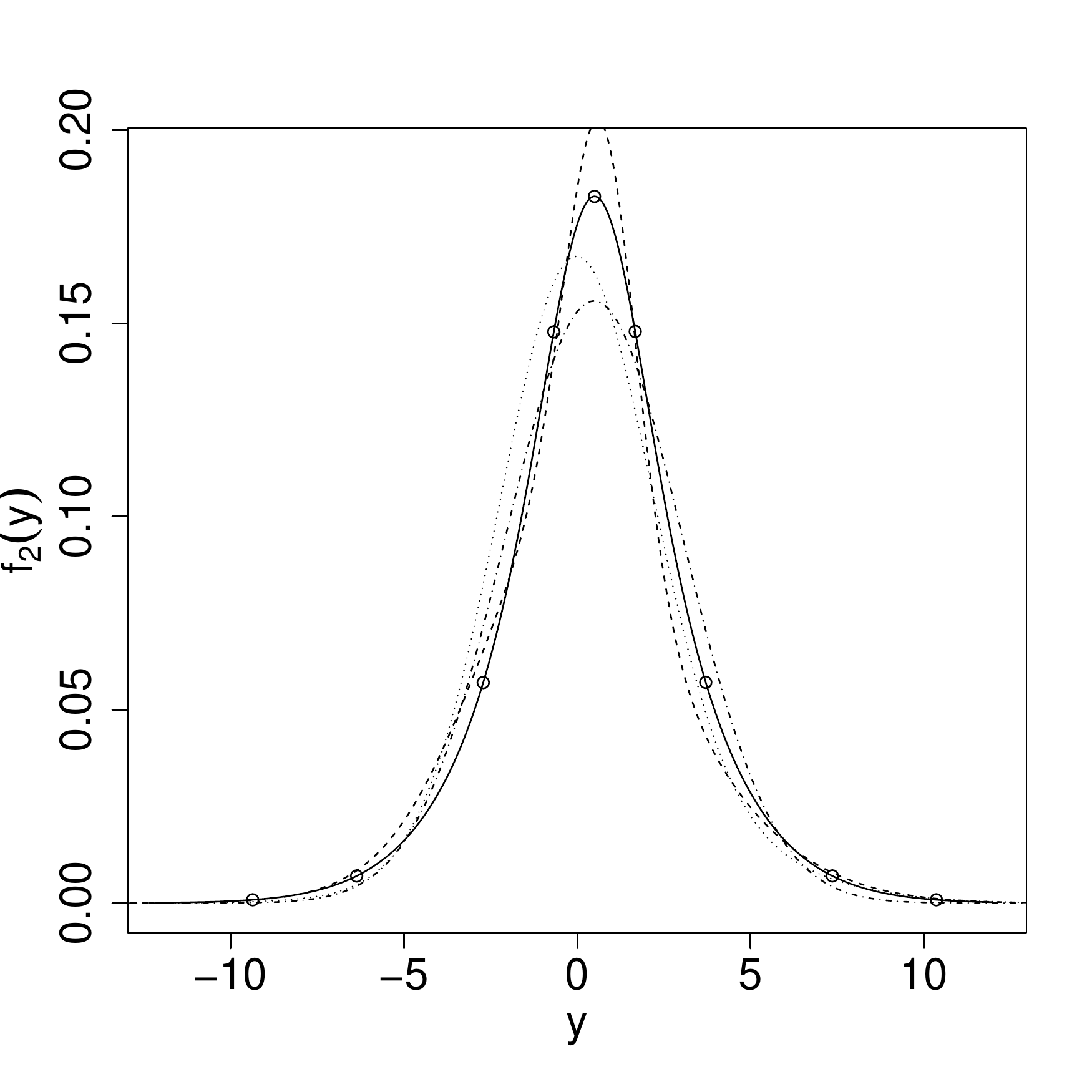}
\includegraphics[width=0.3\textwidth]{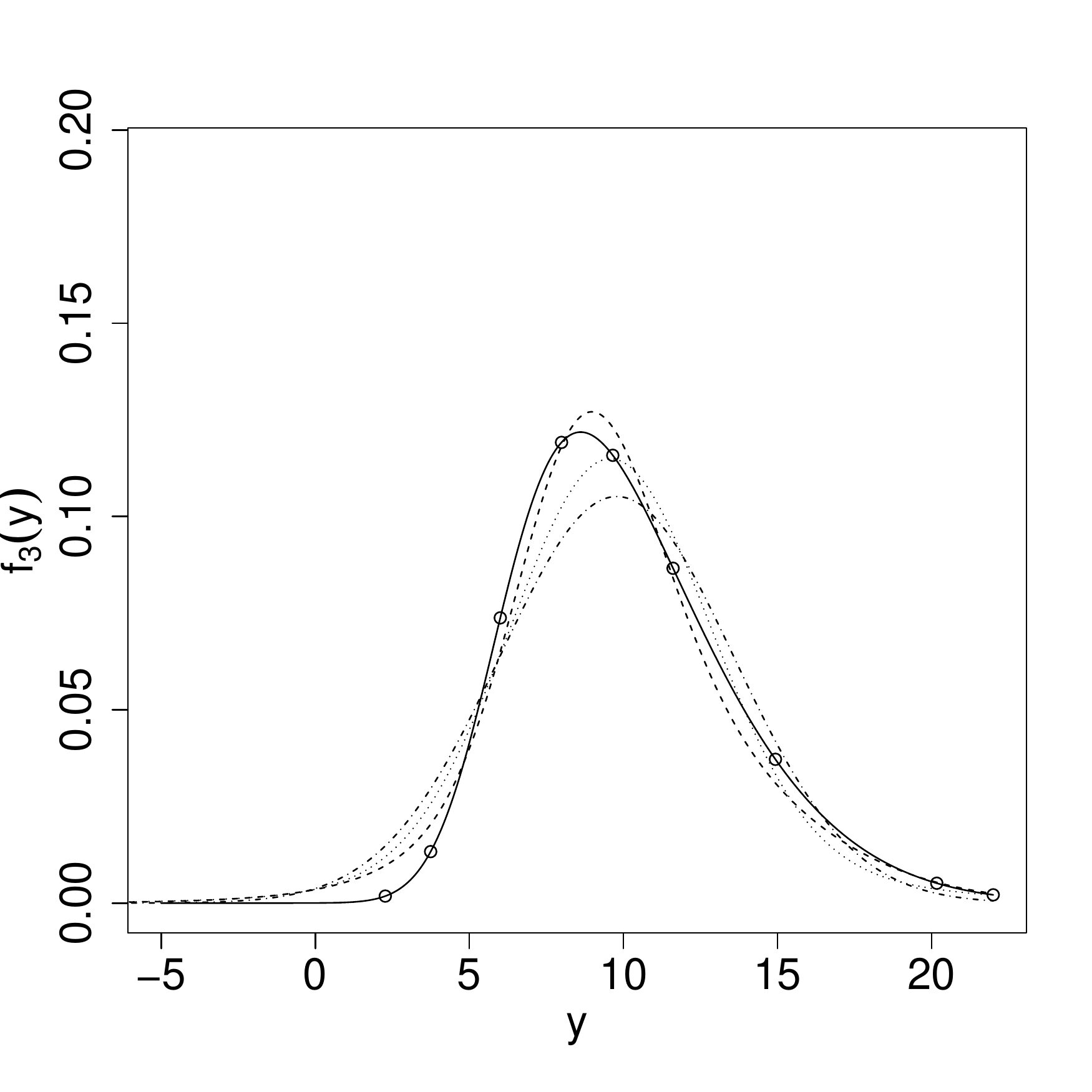}
\caption{State-dependent densities and estimators for a typical sample. Solid line: true densities, dashed line: nonparametric maximum likelihood estimators, dotted line: two-component mixture maximum likelihood estimators, dot-dashed line: Gaussian maximum likelihood estimators }
\label{fig:estimates}
\end{figure}
Figure \ref{fig:estimates} shows the state-dependent normal mixture densities $f_{k,0}$ as well as the fits $f_{\hat{\mu}_{k,n}}$, $f_{\bar{\mu}_{k,n}}$ and $f_{\tilde{\mu}_{k,n}}$ for a typical sample. 
%
%
The nonparametric estimator captures the overall shape of the underlying density, in particular its skewness, much better than both parametric estimators, which deviate substantially from it.   
%
%
%
 \subsection{Goodness of fit test} \label{sec:simtest}
In this section we conduct a formal goodness of fit evaluation for a Gaussian hidden Markov model in the setting of Section \ref{sec:simestimaton}. We use the likelihood ratio test against the nonparametric alternative of state-dependent general Gaussian mixtures as well as against the parametric alternative of state-dependent two-component Gaussian mixtures. 
Critical values are estimated by using the parametric bootstrap. Its consistency requires the asymptotic distribution to depend continuously on nuisance parameters, see \citet{vanderVaart}, and caution is needed in irregular problems, see \citet{Drton}.   

To avoid excessive running times, we estimate critical values under the null model only once. First we use a single long series of the hidden Markov model with parameters as given in Section \ref{sec:simestimaton} and estimate the parameters under the null hypothesis of a Gaussian hidden Markov model, then from this estimated null model we simulate 10000 series of lengths 1000 to obtain the critical values against both classes of alternatives. Each simulated series requires about 85 seconds running time, so that on a computer with 12 processing units a total of 20 hours are required. 
Next, we also simulate from the model in Section \ref{sec:simestimaton} 10000 series of lengths 1000 and use the simulated critical values to estimate the power.

The results for three significance levels are shown in Table \ref{tab:test}. 
\begin{table}[t]
\begin{tabular}{ccc}
 & Parametric vs two-component mixture & Parametric vs nonparametric \\ 
Critical value (90\%) &  1$\cdot$72  &2$\cdot$42 	 \\ 
Simulated power & 95$\cdot$67 & 96$\cdot$44 \\ 
\\
Critical value (95\%)  &  2$\cdot$57 & 3$\cdot$57 \\ 
Simulated power & 92$\cdot$07 & 93$\cdot$97 \\ 
\\
Critical value (99\%)  & 4$\cdot$33  & 6$\cdot$03 \\ 
Simulated power  & 79$\cdot$83 & 85$\cdot$27 \\
\end{tabular} 
\caption{Simulated critical values and powers of the likelihood ratio tests }
\label{tab:test}
\end{table}
Although the critical values for the nonparametric test are larger, it still has slightly higher power at all three significance levels. 
\section{Discussion}\label{sec:discussion}
We obtain nonparametric identification for hidden Markov models under assumptions that are close to minimal. In particular, linear independence of the state-dependent distributions is not required; they are merely assumed to be distinct. By example, we show the necessity of a full rank of the transition probability matrix. Ergodicity of $\Gamma$,  which is assumed in most statistical estimation procedures in a parametric framework, is equivalent to irreducibility and aperiodicity. The proof of Theorem \ref{cor:identagainstgeneralparstat} does not require aperiodicity, so the conclusion holds without it, while the proof of Theorem \ref{the:identgeneral} does require ergodicity. Irreducibility could potentially be dropped by considering communicating classes. 

The majority of hidden Markov models used in applications have parametric state-dependent distributions. However, the need for distributions more flexible than the normal has been recognized in various papers. A popular alternative class are finite mixtures of normals with a given maximal number of components, see \citet{Volant2013} and  \citet{HolzmannSchwaiger2013}. 
The fact that the nonparametric maximum likelihood estimator also has a finite number of components, although potentially growing with the sample size, makes the use of finite mixtures as state-dependent distributions even more attractive. 
Further, as demonstrated in the simulation section, comparison of parametric with nonparametric fits may be used for goodness of fit assessments, both formally by employing likelihood ratio tests and by visually comparing the parametric and nonparametric density estimates.

\section*{Acknowledgement}
The authors would like to thank the editor, the associate editor and three reviewers for helpful comments which lead to improved contents and presentation. Anna Leister and Hajo Holzmann gratefully acknowledge financial support of the ``Deutsche Forschungsgemeinschaft''. 
\section*{Supplementary material}
\label{SM}
Supplementary material available at \Bka\ online includes proofs of the results and additional simulations for a scenario with linearly dependent state-dependent distributions. 

\appendix


\section{Outline of the proofs}

We present outlines of the proofs of the theorems. 

\begin{proof}[{\it Proof of Theorem \ref{cor:identagainstgeneralparstat}}]
The proof follows that of Theorem 1 in \citet{gassiat}, which in turn combines arguments in \citet{allman} for generic identification of hidden Markov models and finite mixtures of product distributions. 

In the first step, for $T \geq K-1$ we form the blocks 
\[ V_T = Y_1^T = \big(Y_1, \ldots, Y_T\big), \qquad W_T = Y_{T+2}^{2T+1} = \big(Y_{T+2}, \ldots, Y_{2T+1}\big),\]
and show that the conditional distributions 
\[ G_T(y_1^T;k) = \text{pr}\,\big(W_T \leq y_1^T \mid X_{T+1} = k\big) \quad (k=1, \ldots, K)\]
are linearly independent, and so are 
\begin{equation}\label{eq:conddistrfct}
 H_T(y_1^T;k) =  \text{pr}\,\big(V_T \leq y_1^T \mid X_{T+1} = k\big)\quad (k=1, \ldots, K).
\end{equation}
This is the crucial non-obvious step in our setting, since the state-dependent distribution functions $F_1, \ldots, F_K$ may be linearly dependent, and it requires some technical effort. It is essential to make the arguments in  \citet{allman} work. 

In the second step, we follow \citet{allman}, who rely on Theorem 4a of 
\citet{kruskal} and conclude that for $T \geq K-1$  the distribution functions $H_T(; k), F_k, G_T(;k)$ ($k=1, \ldots, K$) are identified up to joint label swapping. The linear independence from step 1 together with the assumption that the $F_k$ are all distinct will result in a sum of Kruskal ranks not less than $2K+2$, see the supplement, as required for the argument in this step.

In the third step, we relate the identified distributions $G_{K-1}(\cdot;k)$ and $G_{K}(\cdot;l)$ ($k,l=1, \ldots, K$) via the transition probability matrix $\Gamma$ which will thus also be identified.  
\end{proof}

\begin{proof}[{\it Proof of Theorem \ref{the:identgeneral}}]

To identify the $ H_T(\cdot;k)$ $(k=1, \ldots, K)$ in (\ref{eq:conddistrfct}), we consider the time reversal 
\[ \big\{(X_{T+1}, Y_{T+1}), \ldots,  (X_{1}, Y_{1})\big\},\]
which is a segment of a hidden Markov model with inhomogeneous underlying Markov chain and state-dependent distributions $F_1, \ldots, F_K$, the Markov chain starting in $\lambda \Gamma^T$. For technical reasons, we first require that $\Gamma$ and $\lambda$ have only positive entries, and then relax these assumptions by using higher transitions of order  $t_0 = K^2-2K+2$ which results in a transition probability matrix $\Gamma^{t_0}$ which has strictly positive entries \citep{holladay}, as well as by starting at time $t_0$, which results in a starting distribution $\lambda \Gamma^{t_0}$ with positive entries. 
\end{proof}

\begin{proof}[{\it Proof of Theorem \ref{the:mlcontrast}
}]
The existence of the limit as well as its independence from the starting distributions may be deduced from Kingman's subadditive ergodic theorem, as shown in \citet{ler}. To show definiteness, from the construction in \citet{ler}, letting
\[ {\Delta}^{K-1} = \big\{(s_1, \ldots, s_K) \in [0,1]^K:\quad s_1 + \cdots + s_K=1 \big\}\]
denote the $(K-1)$-dimensional unit simplex, one obtains a probability measure $Q$ on ${\Delta}^{K-1} \times {\Delta}^{K-1}$ such that for $T \geq 2$, 
\begin{align}\label{eq:Kulbackdiff}
\begin{split}
& T\, K\{(\Gamma_0, f_{1,0}, \ldots, f_{K,0}),(\Gamma, f_1, \ldots, f_K) \}\\
= & \int  \int  g_T(y_1^T;u, \Gamma_0, f_{1,0}, \ldots, f_{K,0})
\log \Big\{\frac{g_T(y_1^T;u, \Gamma_0, f_{1,0}, \ldots, f_{K,0})}{g_T(y_1^T;v, \Gamma, f_{1}, \ldots, f_{K})} \Big\}\, d\nu^{\otimes T}(y_1^T)\, dQ(u,v).
\end{split}
\end{align}
The inner integral corresponds to the ordinary Kullback--Leibler divergence of the distribution of the segments $(Y_1, \ldots, Y_T)$ from two hidden Markov models with parameters $ u, \Gamma_0, f_{1,0}, \ldots, f_{K,0}$ and $v, \Gamma, f_{1}, \ldots, f_{K}$, $u$ and $v$ denoting the starting distributions. Non-negativity is then obvious. To show definiteness, choose $T = (2 K+1) (K^2-2K+2) + 1$. From Theorem \ref{the:identgeneral}, which implies identification with arbitrary starting distributions, it follows that for distinct parameters $ \Gamma_0, f_{1,0}, \ldots, f_{K,0}$ and $ \Gamma, f_{1}, \ldots, f_{K}$, the inner integral is strictly positive for any values of $u$ and $v$, and hence so is (\ref{eq:Kulbackdiff}). 
\end{proof}

\begin{proof}[{\it Proof of Theorem \ref{prop:existencemle}}]
This follows using arguments from convex analysis  similar to those in \citet{lindsay}. 
\end{proof}

\begin{proof}[{\it Proof of Theorem \ref{th:consistency}
}]
To prove the theorem we may follow the arguments in \citet{ler} for the parametric case to obtain the consistency of $\hat \Gamma_n$ as well as $d_{w}\big(\hat \mu_{k,n}, \tilde \Theta_{k,0} \big) \to 0$ in probability, where
\[ \tilde \Theta_{k,0} = \big\{ \mu \in \tilde \Theta:\quad f_\mu = f_{\mu_{k,0}} \big\} \quad (k=1, \ldots, K).\]
The main additional issue is to conclude that $f_{\hat \mu_{k,n}}(y) \to f_{k,0}(y)$ if $\tilde \Theta_{k,0}$ contains more than a single mixing distribution. Here for fixed $y \in \cS$ we use approximation of $\vartheta \mapsto f_\vartheta(y)$  by Lipschitz-continuous functions and the bounded Lipschitz metric on $\tilde \Theta$.
\end{proof}

\bibliographystyle{ims}
\bibliography{db}


\end{document}


\LARGE
\vspace*{0.5cm}
\begin{center}
{\bfseries Supplementary material for \\Nonparametric identification and maximum likelihood estimation for hidden Markov models\\*[0.4cm]}
\large
\textbf{Grigory Alexandrovich, Hajo Holzmann\footnote{Address for correspondence: Prof.~Dr.~Hajo Holzmann, 
Philipps-Universität Marburg,
Fachbereich Mathematik und Informatik, 
Hans-Meerweinstr. 
D-35032 Marburg, Germany
email: holzmann@mathematik.uni-marburg.de,
Fon: + 49 6421 2825454
} and Anna Leister}\\*[0.4cm]
{\sl Fakult\"at f\"ur Mathematik und Informatik, Philipps-Universit\"at Marburg, Germany}
\end{center}
\normalsize
\section{Proofs of the identification results}


\subsection{Proofs of the main results}
For convenience, we recall the notation, the assumptions and the statements of the theorems. 
A discrete-time hidden Markov model consists of an observed process $(Y_t)_{t \in \N}$ and a latent, unobserved process $(X_t)_{t \in \N}$, such that 
first, the $Y_t$ are independent given the $X_t$, second, the conditional distribution of $Y_s$ given the $X_t$ depends on $X_s$ only and third, $X_t$ is a finite-state Markov chain. We assume that $X_t$ is time-homogeneous. 
The cardinality $K$ of the state space of $X_t$ is called the number of states. The conditional distributions of $Y_s$ given $X_s = k$ ($k=1, \ldots, K$), are called the state-dependent distributions, and  we assume that they are independent of $s$. The entries of the transition probability matrix are denoted by $\Gamma= (\alpha_{j,k})_{j,k=1, \ldots, K}$. Further, assume that the $Y_t$ take values in any subset of Euclidean space $\cS \subset \R^q$, and denote the distribution functions of the state-dependent distributions by $F_k$ $(k=1, \ldots, K)$. 

\begin{assumption}\label{eq:tpm}
The transition probability matrix $\Gamma= (\alpha_{j,k})_{j,k=1, \ldots, K}$ of $(X_t)$ has full rank and is ergodic.%
\end{assumption}

\begin{assumption}\label{eq:statedep}
The state-dependent distributions $F_k$ ($k=1, \ldots, K$) are all distinct.
\end{assumption}

\begin{assumption}\label{eq:stationary}
The Markov chain $(X_t)$ is stationary with starting distribution $\pi$, the stationary distribution of $\Gamma$. 
\end{assumption}

\begin{theorem}\label{cor:identagainstgeneralparstat}
For given $K$, let $\Gamma$, $F_1, \ldots, F_K$ and $\tilde \Gamma$, $\tilde F_1, \ldots, \tilde F_K$ be two sets of parameters for a $K$-state hidden Markov model, such that the joint distributions of $\big(Y_1, \ldots, Y_{2 K +1}\big)$ under both sets of parameters are equal. Further, suppose that $\Gamma$ and $F_1, \ldots, F_K$ satisfy Assumptions \ref{eq:tpm}--\ref{eq:stationary}. Then both sets of parameters coincide up to label swapping. 
\end{theorem}

In order to keep the arguments as transparent as possible, we first prove the following result. 

\begin{proposition}\label{the:firstthident}
Suppose that for a known number of states $K$, Assumptions \ref{eq:tpm}--\ref{eq:stationary} are satisfied. Then the parameters $\Gamma$ and $F_1, \ldots, F_K$ are identified from the joint distribution of 
$ \big(Y_1, \ldots, Y_{2 K +1}\big)$
up to label swapping. 
\end{proposition}

This proposition states that for given $K$, the parameters $\Gamma$ and $F_1, \ldots,F_K$ are identified within the class of parameters satisfying Assumptions \ref{eq:tpm}--\ref{eq:stationary}. However, from the proofs and exploiting the full strength of Theorem 4a in \citet{kruskal}, we also obtain Theorem \ref{cor:identagainstgeneralparstat}, as shown below. 

Before we turn to the proof of Proposition \ref{the:firstthident}, we introduce some notation and recall a result of \citet{kruskal}. For a vector space $V$ we let $\text{dim}(V)$ denote its dimension. For vectors $v_1, \ldots, v_n \in V$ we let $\text{span}\,\{v_1, \ldots, v_n\}$ denote the subspace of $V$ spanned by $v_1, \ldots, v_n$. Further, for numbers $x_1, \ldots, x_n \in \R$ we let $\text{diag}\,(x_1, \ldots, x_n)$ denote the $n$-dimensional diagonal matrix with entries $x_1, \ldots, x_n$. We let $1_K = (1, \ldots, 1) \in \R^K$ and $I_K = \text{diag}(1_K)$ denote the $K$-dimensional unit matrix. 
For vectors $z_1, z_2 \in \R^T$ we write $z_1 \leq z_2$ if this holds for each coordinate. 
For a matrix $M$ we let $M'$ denote its transpose. 
For given matrices $M_i \in \R^{K \times n_i}$ ($n_i \in \N$; $i=1,2,3$) let $[M_1 \, M_2]$ denote the $K \times (n_1 + n_2)$ block matrix and let 
\begin{equation}\label{eq:threewayarray}
 \big<M_1, M_2, M_3 \big>(i_1,i_2,i_3) = \sum_{k=1}^K (M_1)_{k,i_1} \, (M_2)_{k,i_2}\, (M_3)_{k,i_3}\quad (i_j=1,\ldots, n_j),
\end{equation}
a three-way array. The Kruskal rank of a matrix $M \in \R^{K \times n}$, denoted rank$_{Kr} M$, is the maximal $j$ ($j\in\{0,\hdots,K\}$), for which each set of $j$ rows in $M$ are linearly independent (as vectors in $\R^n$). 
Then Theorem 4a in 
\citet{kruskal} states that if $M_i, N_i \in \R^{K \times n_i}$ ($n_i \in \N$; $i=1,2,3$) are two sets of real matrices such that 
\[ \big<M_1, M_2, M_3 \big> = \big<N_1, N_2, N_3 \big>\]
and
\[ \text{rank}_{Kr} \, M_1  + \text{rank}_{Kr} \, M_2 + \text{rank}_{Kr} \, M_3 \geq 2 K +2,\]
then there exists a permutation matrix $P$ and diagonal matrices $\Lambda_i$, such that $\Lambda_1 \Lambda_2 \Lambda_3 = I_K$ and
$ N_i = \Lambda_i \, P {M}_i$ ($i=1,2,3$). \hfill $\diamond$


\begin{proof}[{\it { Proof of Proposition \ref{the:firstthident}}}]\qquad 
\begin{step}[Blocks in the joint distribution and linear independence of conditional distributions]\label{step1prop1}
For $T \geq K-1$ consider 
\[ V_T = Y_1^T = \big(Y_1, \ldots, Y_T\big) \text{  and  } W_T = Y_{T+2}^{2T+1} = \big(Y_{T+2}, \ldots, Y_{2T+1}\big).\]
The conditional distribution functions of $W_T$ given $X_{T+1} = k$ ($k=1, \ldots, K$), are given by
\begin{align*}\label{eq:conddistr}
G_T(y_1^T;k)  = \text{pr}\,\big(W_T \leq y_1^T \mid X_{T+1} = k\big)
   = \sum_{k_1 = 1}^K \cdots \sum_{k_T = 1}^K \alpha_{k, k_1}\, \prod_{t=2}^T \alpha_{k_{t-1}, k_{t}}\, \prod_{t=1}^T F_{k_t} (y_t).
\end{align*}
From Lemma \ref{lemma:indone} below we have that $G_T(\cdot;k)$ ($k=1, \ldots, K$) are linearly independent functions on $\cS^T$ and furthermore, there exist $z_1, \ldots, z_K \in \cS^T$ for which the $K \times K$-matrix
\[ A_1 = \big\{G_T(z_t; k )\big\}_{k,t=1, \ldots, K}\]
has full rank $K$. Here, $k$ is the row index and $t$ the column index. Further, consider the time reversal 
\[ \tilde \Gamma = \big(\tilde \alpha_{j,k}\big)_{j,k=1, \ldots, K}, \quad \tilde \alpha_{j,k} = \frac{\pi_k \alpha_{k,j}}{\pi_j}.\]
Then for $(y_T, \ldots, y_1) \in \cS^T$ we have that
\begin{align*}
 H_T(y_T, \ldots, y_1;k) & = \text{pr}\,\big\{V_T \leq (y_T, \ldots, y_1) \mid X_{T+1} = k\big\}\\
& =  \sum_{k_1 = 1}^K \cdots \sum_{k_T = 1}^K \tilde \alpha_{k, k_1}\, \prod_{t=1}^{T-1} \tilde \alpha_{k_t, k_{t+1}}\, \prod_{t=1}^T F_{k_t} (y_t).
\end{align*}
Applying Lemma \ref{lemma:indone} with $\tilde \Gamma$,  we conclude that $H_T(\cdot;k)$ ($k=1, \ldots, K$) are linearly independent functions on $\cS^T$ and furthermore, there exist $\tilde z_1, \ldots, \tilde z_K \in \cS^T$ for which we have the rank $K$ matrix
\begin{equation}\label{eq:hilfmatrixagain}
 A_2 = \big\{H_T(\tilde z_t; k )\big\}_{k,t=1, \ldots, K}.
\end{equation}
\end{step}
\begin{step}[Kruskall's theorem and identification of conditional distributions]\label{step2prop1}
In this step we show that under Assumptions \ref{eq:tpm} and \ref{eq:statedep}, for $T \geq K-1$  the distribution functions $H_T(; k), F_k, G_T(;k)$ ($k=1, \ldots, K$) are identified up to joint label swapping. 

Let $z, \tilde z \in \cS^T$ and $y \in \cS$ be arbitrary points. Set $m= K (K-1)/2$. From Lemma \ref{lemma:distinctKruskal} below there exist points $y_j \in \cS$ ($j=1, \ldots, m$), such that the $K \times (m+2)$-matrix 
\begin{equation*}\label{eq:matrix2}
M_2 = \big[ \{ F_i(y_j) \}_{i=1,\hdots,K; j=1,\hdots,m}, \{ F_i(y) \}_{i=1,\hdots,K}, 1_K \big]
\end{equation*}
has Kruskal rank at least 2. From step \ref{step1prop1} the $K \times (K+2)$-matrices
\begin{align*}
\begin{split}
M_3  & = \left[ A_1, \{ G_T(z;k) \}_{k=1,\hdots,K}, 1_K\right], \qquad
M_1  = \left[ A_2, \{ H_T(\tilde z;k) \}_{k=1,\hdots,K}, 1_K\right],\\
\tilde{M}_1 &= \textnormal{diag}(\pi)M_1,
\end{split}
\end{align*}
have full rank $K$, where we use $\pi_k > 0$ ($k=1,\ldots,K$) for $\tilde{M}_1$, and therefore 
\begin{align}
\label{kruskalr}
\text{rank}_{Kr}(\tilde{M_1}) + \text{rank}_{Kr}(M_2) + \text{rank}_{Kr}(M_3) = 2K + 2. 
\end{align}

Now we show that the three-dimensional array
\[ M = \big<\tilde{M}_1, M_2, M_3\big>\]
as defined in (\ref{eq:threewayarray}), %
is identified from the joint distribution of $Y_1^{2T +1}$. 
In the following, we write $z_{K+1} = z$, $\tilde z_{K+1} = \tilde z$, $y_{m+1} = y$. We have that 
\begin{align}\label{eq:theformulakrus1}
\begin{split}
& M(i,j,r) =  \sum_{k = 1}^K \pi_k\, H_T(\tilde z_i; k)  F_k(y_j) G_T(z_r; k) \\
 &= \sum_{k = 1}^K \pi_k \text{pr}\,(Y_1^T \leq \tilde z_i \mid X_{T+1} = k) \text{pr}\,(Y_{T+1} \leq y_j \mid X_{T+1} = k)\text{pr}\,(Y^{2T+1}_{T+2} \leq z_r \mid X_{T+1} = k)\\
 &= \sum_{k = 1}^K\pi_k \text{pr}\,(Y_1^T \leq \tilde z_i, Y_{T+1} \leq y_j ,  Y^{2T+1}_{T+2} \leq z_r \mid X_{T+1} = k) \\
 &= \text{pr}\,(Y_1^T \leq \tilde z_i, Y_{T+1} \leq y_j ,  Y^{2T+1}_{T+2} \leq z_r)\quad (1 \leq i;\ r \leq K+2;\ j= 1,\hdots,m+2).
\end{split}
\end{align}
%

%
Similarly, 
\begin{align}
\label{eq:theformulakrus2}
M(K+2,j,r) &= \text{pr}\,(Y_{T+1} \leq y_j ,  Y^{2T+1}_{T+2} \leq z_r),  & M(K+2,m+2,r) &= \text{pr}\,(Y^{2T+1}_{T+2} \leq z_r),\notag\\
M(i,m+2,r) &=  \text{pr}\,(Y_1^T \leq \tilde z_i ,  Y^{2T+1}_{T+2} \leq z_r), & M(K+2,j,
K+2) &= \text{pr}\,(Y_{T+1} \leq y_j),\\
M(i,j,K+2) &=  \text{pr}\,(Y_1^T \leq \tilde z_i, Y_{T+1} \leq y_j),  &
M(i,m+2,K+2) &= \text{pr}\,(Y_1^T \leq \tilde z_i),\notag
\end{align}
and $M(K+2, m+2, K+2) = 1$. Evidently, these quantities are identified from the distribution of $Y_1^{2T+1}$. 

Now, using (\ref{kruskalr}) we apply Theorem 4a in 
\citet{kruskal} to show that the matrices $\tilde{M}_1, M_2$ and $M_3$ are identified from $M$ up to scaling and permutation, that is there exist a permutation matrix $P$ and diagonal matrices $\Lambda_1, \Lambda_2, \Lambda_3$, such that $\Lambda_1 P \tilde{M}_1, \Lambda_2PM_2, $ and $\Lambda_3PM_3$ are known and the relationship $\Lambda_1 \Lambda_2 \Lambda_3 = I_K$ holds.
Since we know that in the last column of $M_2$ there are only ones, we obtain the $i$th diagonal element of the scaling matrix $\Lambda_2$ as $(\Lambda_2 P M_2)_{i, K+2}$ ($i=1,\ldots,K$). Similarly we find the matrix $\Lambda_3$. The elements of $\Lambda_1$ can then be determined by the relationship $\Lambda_1 \Lambda_2 \Lambda_3 = I_K$. Hence we identified the matrices  $\tilde{M}_1, M_2$ and $M_3$ up to simultaneous row permutations and therefore the values 
$H_{T}(\tilde z;k), F_k(y), G_T(z;k)$ at arbitrary points $z, \tilde z \in \cS^T$ and $y \in \cS$. 
Finally, we show that for distinct values of $z, \tilde z \in \cS^T$ and $y \in \cS$, the matrices $P$ and $\Lambda_3$ remain the same, so that there is a joint label swapping. Suppose that for distinct values, we get $\tilde P$ and $\tilde \Lambda_3$, 
The matrix $[A_1, 1_K]$, which is the submatrix of both versions of $M_3$ consisting of the first $K$ columns and the last column, we obtain 
\[ \Lambda_3 P [A_1, 1_K] = \tilde \Lambda_3 \tilde P [A_1, 1_K].\]
As above, since the last column of $ P [A_1, 1_K] = \Lambda_3^{-1}\, \tilde \Lambda_3 \, \tilde P [A_1, 1_K]$ equals $1_K$ as well as the diagonal entries of $\Lambda_3^{-1}\, \tilde \Lambda_3$, we get $\Lambda_3^{-1}\, \tilde \Lambda_3 = I_K$ and hence $P [A_1, 1_K] = \tilde P [A_1, 1_K] $. Since $[A_1, 1_K]$ has full row rank $K$, we get $P = \tilde P$, as required. 
\end{step}

\begin{step}[Identification of $\Gamma$]\label{step3prop1}
It remains to identify the transition probability matrix $\Gamma$. We choose $T=K-1$, and after applying the result in step \ref{step2prop1}, fix a labeling $H_T(\cdot; k), F_k, G_T(\cdot;k)$ ($k=1, \ldots, K$). 
For $z_1, \ldots, z_K \in \cS^T$ as in step \ref{step1prop1} and $y \in \cS$ we consider the $K \times K$-matrix 
\[ A = \big[G_{T+1}\{(y,z_t); k \}\big]_{k,t=1, \ldots, K}.\]
From step \ref{step2prop1},  $H_{T+1}(\cdot; k), F_k, G_{T+1}(\cdot;k)$ are identified up to joint label swapping and hence so is the matrix $A$. 
Since the $F_k$ are all distinct, we may choose the same labeling as the one fixed for $H_T(\cdot; k), F_k, G_T(\cdot;k)$ ($k=1, \ldots, K$). In this case, we have that
\[ A = \Gamma \,\textnormal{diag}\big\{F_1(y), \ldots, F_K(y) \big\}\, A_1, \]
where $A_1$ is defined in step \ref{step1prop1}. Now choose $y$ large enough so that $F_k(y) \not=0$ ($k=1, \ldots, K$), so that $\Gamma$ is identified as
\[ \Gamma=A \, A_1^{-1} \, \textnormal{diag}\big[\{F_1(y)\}^{-1}, \ldots, \{F_K(y)\}^{-1} \big], \]
which concludes the proof.
\end{step}
\end{proof}

\begin{proof}[{\it {Proof of Theorem \ref{cor:identagainstgeneralparstat}}}]

The sets of parameters are denoted by $\Gamma$ and $F_1, \ldots, F_K$ with stationary starting distribution $\pi$, and $\tilde \Gamma$ and $\tilde F_1, \ldots, \tilde F_K$ with arbitrary starting distribution $\lambda$. 

Step \ref{step1prop1} in the proof of Proposition \ref{the:firstthident} applies to the parameters $\Gamma$ and $F_1, \ldots, F_K$. We define the matrices $\tilde M_1$, $M_2$ and $M_3$ as in step \ref{step2prop1}, these satisfy (\ref{kruskalr}). Further, the matrices $N_1$, $N_2$ and $N_3$ as the conditional distribution functions of $V_T$, $Y_{T+1}$ and $W_T$ given $X_{T+1}=k$ under the parameters $\tilde \Gamma$, $\tilde F_1, \ldots, \tilde F_K$ and $\lambda$, evaluated at the same points as for $M_1$, $M_2$ and $M_3$. Let $\tilde N_1 = \textnormal{diag}\, (\lambda \tilde \Gamma^{T}) \, N_1$, where we observe that $\lambda \tilde \Gamma^{T}$ is the marginal distribution of $X_{T+1}$ under this parameter set. 

Now, (\ref{eq:theformulakrus1}) and (\ref{eq:theformulakrus2}) show that under the assumption that both sets of parameters induce the same distribution of $Y_1, \ldots, Y_{2 T+1}$, 
\[ \big<\tilde M_1, M_2, M_3  \big> = \big<\tilde N_1, N_2, N_3  \big>.\]
For an application of Theorem 4a in 
\citet{kruskal}, it suffices that the matrices $\tilde M_1$, $M_2$ and $M_3$ satisfy (\ref{kruskalr}), hence there is a $K\times K$ permutation matrix $P$ and diagonal matrices $\Lambda_i$ ($i=1,2,3$) with $\Lambda_1 \Lambda_2 \Lambda_3 = I_K$, such that 
\[ M_i = \Lambda_i P N_i \quad (i=2,3) \text{  and  }  \tilde M_1 = \Lambda_1 P \tilde N_1.\]
Since $M_i, N_i$ ($i=2,3$), have only ones in the last column, $\Lambda_2 = \Lambda_3 = I_K$ and hence also $\Lambda_1 = I_K$. It follows that $N_3$ and $\tilde N_1$ must also have full rank, and that $P$ is uniquely determined, that $\pi = \lambda \tilde \Gamma^T$ which are contained in the last column of $\tilde M_1 = \tilde N_1$, and that the conclusion of step \ref{step2prop1} in Proposition \ref{the:firstthident} holds true. The equality $\Gamma = \tilde \Gamma$ follows as in step \ref{step3prop1}, and since $\Gamma$ is invertible and $\pi \Gamma^{-1} = \pi$, from $\pi = \lambda \Gamma^T$ we obtain $\pi = \lambda$.  
\end{proof}

We let $\lambda$ denote an arbitrary $K$-state probability vector. 
\begin{theorem}\label{the:identgeneral}
For a known number of states $K$, let $\lambda, \Gamma$, $F_1, \ldots, F_K$ and $\tilde \lambda, \tilde \Gamma$, $\tilde F_1, \ldots, \tilde F_K$ be two sets of parameters for a $K$-state hidden Markov model, such that the joint distributions of $ \big(Y_1, \ldots, Y_{T}\big)$
with $T = (2 K+1) (K^2-2K+2) + 1$, are equal under both sets of parameters. Further, suppose that $\Gamma$ and $F_1, \ldots, F_K$ satisfy Assumptions \ref{eq:tpm} and \ref{eq:statedep}. Then both sets of parameters coincide up to label swapping. 
\end{theorem}

\begin{proof}\qquad 
\setcounter{step}{0}
\begin{step}[First assume that $\lambda$ has only positive entries]\label{step1th2}
 We shall show that in this case, from the joint distribution of $ \big(Y_1, \ldots, Y_{2 K +1}\big)$ we identify $\Gamma$ and $F_1, \ldots, F_K$, and the conditional distributions
\begin{align*}
 H_T(y_1^T;k) = \text{pr}\,\big(Y_1^T\leq y_1^T\, \mid \, X_{T+1} = k\big) \qquad (k=1, \ldots, K;\ T=K-1,K),
\end{align*}
up to label swapping. To this end we may follow the proofs of Proposition \ref{the:firstthident} and Theorem \ref{cor:identagainstgeneralparstat}, and it remains to show that the distribution functions $H_T(\cdot ;k)$ ($k=1, \ldots, K$), are linearly independent, where $T=K - 1$. The time reversal 
$ (X_{T+1}, \ldots, X_1)$
is an inhomogeneous Markov chain, and therefore 
\[ \big\{(X_{T+1}, Y_{T+1}), \ldots,  (X_{1}, Y_{1})\big\}\]
 is a hidden Markov model with inhomogeneous underlying Markov chain and state-dependent distributions $F_1, \ldots, F_K$. In particular, 
\[ \lambda^{(t)} = \lambda \Gamma^{t-1},\qquad  \big(\tilde\Gamma^{(t)}\big)_{i,j} = \frac{\lambda^{(t)}_j \alpha_{j,i}}{\sum_{k=1}^K \lambda^{(t)}_k \alpha_{k,i}} = \big(\tilde \alpha_{i,j}^{(t)} \big)_{i,j=1, \ldots, K}\quad (t=1, \ldots, T)\]
and we have that
\[ H_T(y_1^T;k) = \sum_{k_1 = 1}^K \cdots \sum_{k_T = 1}^K \tilde \alpha_{k, k_T}^{T}\, \prod_{t=2}^T \tilde \alpha_{k_{t}, k_{t-1}}^{(t-1)}\, \prod_{t=1}^T F_{k_t} (y_t).\] 
Since all entries in $\lambda$ are strictly positive, the matrices $\tilde\Gamma^{(t)}$ ($t=1, \ldots, T$) all have full rank. The argument in the proof of Lemma \ref{lemma:indone} applies to show that the $H_T(\cdot ;k)$ ($k=1, \ldots, K$), are linearly independent. 
\end{step}
\begin{step}[ Both $\Gamma$ and $\lambda$ have only strictly positive entries]\label{step2th2}

We show that in this case all parameters $\lambda$, $\Gamma$ and $F_1, \ldots, F_K$ are identified from the joint distribution of 
$ \big(Y_1, \ldots, Y_{2 K +1}\big)$.
It remains to identify $\lambda$. We may argue similarly as in step \ref{step3prop1} of Proposition \ref{the:firstthident}. For $T=K-1$, we may identify both $H_T(\cdot; k)$ and $H_{T+1}(\cdot;k)$, where we have chosen a fixed equal labelling for both distribution functions.   
Again, we find $\tilde z_1, \ldots, \tilde z_K \in \cS^T$ such that the identified $K \times K$-matrix $A_2$ in (\ref{eq:hilfmatrixagain}) 
has full rank $K$ in this situation as well. For $y \in \cS$ consider the identified $K \times K$-matrix 
\begin{align*}
 \big[H_{T+1}\{(\tilde z_t, y); k\} \big]_{k,t=1, \ldots, K}\ = 
\tilde \Gamma^{(T+1)}\,  \textnormal{diag}\big\{F_1(y), \ldots, F_K(y) \big\} A_2, 
\end{align*}
which, for $y$ large enough so that $F_k(y) \not=0$ ($k=1, \ldots, K$), allows to identify $\tilde \Gamma^{(T+1)}$. 
Therefore, for each $j$, we identify
\[ \frac{\tilde \alpha_{j,i}^{(T+1)}}{\alpha_{i,j}} = \frac{\lambda_i^{(T+1)}}{c_j}\qquad (i=1, \ldots, K),\]
where $c_j$ is a positive constant. If we fix $j$, this identifies $\lambda^{(T+1)}$ up to scale. Since $\lambda^{(T+1)}$ is a probability vector, it is identified and since $\Gamma$ is invertible and identified and $\lambda^{(T+1)} = \lambda \Gamma^{T}$, $\lambda$ itself is identified.   
\end{step}
\begin{step}[Conclusion of the proof]
Now we conclude the proof of the theorem. 
Let $t_0 = K^2-2K+2$. Then from \citet{holladay}, $\Gamma^{t_0}$ has strictly positive entries. 
Observe that $\big(Y_{t_0+1}, \ldots, Y_{t_0+ 2K+1}\big)$ is a segment of a hidden Markov model with starting vector $\lambda \Gamma^{t_0}$, which has only positive entries.  Using step \ref{step1th2} we therefore identify $\Gamma$ and $F_1, \ldots, F_K$. Then, using the result in step \ref{step2th2}, from
\[\big(Y_{t_0+1},Y_{2\,t_0+1}, \ldots, Y_{(2K+1)\,t_0 + 1}\big),\]
which is a segment of a hidden Markov model where the Markov chain starts in $\lambda \Gamma^{t_0}$ and has transition probability matrix $\Gamma^{t_0}$, and the state-dependent distributions are $F_1, \ldots, F_K$, we identify $\tilde \lambda = \lambda \Gamma^{t_0}$, and therefore also $\lambda = \tilde \lambda \Gamma^{-t_0}$.  
\end{step}
\end{proof}


\subsection{Technical lemmas for the identification proofs}

\begin{lemma}
\label{lemma:distinctKruskal}
Let $G_k$\ ($k=1, \ldots, K$) be distinct distribution functions. Then there exist $y_1\ldots,y_m \in \cS$ where $m = K\,(K-1)/2$ such that the $K \times (m+1)$ matrix 
\[ \left[ \{G_i(y_j)\}_{i=1,\hdots,K; \ j=1,\hdots,m}, 1_K \right]\]
has Kruskal rank at least two.
\end{lemma}

\begin{proof}
There exists a $y_{i,j} \in {\cS}$ such that $G_i(y) \neq G_j(y)$ ($i,j\in\{1,\hdots,K\}$; \ $i<j$). Let $y_1,\ldots,y_m$ be points corresponding to the $m$ pairs of indices. 
\end{proof}

The next lemma is the key technical result.  

\begin{lemma}\label{lemma:dimenextend}
Let $t \leq K-1$ and $v_1, \ldots, v_t \in \R^K$ be linearly independent vectors. Assume that the entries of $v_1$ are all strictly positive. Let $\Gamma$ be a $K\times K$ stochastic matrix of full rank and let $F_1, \ldots, F_K$ be distinct distribution functions. 
Then there exist $y \in \cS$ and a $j \in \{1, \ldots, t\}$ for which, letting 
\[ D_y = \textnormal{diag}\, \big\{F_1(y), \ldots, F_K(y)\big\},\]
the $K \times (t+1)$-matrix
\[ \big[\Gamma v_1, \ldots, \Gamma v_t, D_y\Gamma v_j \big]\]
has full rank $t+1$. 
\end{lemma}

\begin{proof}
First, we can construct vectors $o^{(1)}, \ldots, o^{(K-t)} \in \R^K$ orthogonal to \linebreak $\textnormal{span}\,  \{ \Gamma v_1, \ldots, \Gamma v_t\}$, which are of the form
\begin{equation*}
o^{(i)} = [o_1^{(i)}, \ldots, o_t^{(i)}, 0, \ldots,0, -1,0, \ldots, 0 ] \quad (i=1, \ldots, K-t),
\end{equation*}
where the $-1$ is at the $(t+i)$th place, after possibly relabeling the coordinates of $\R^K$. Indeed, observe that the $K \times t$ matrix $\Gamma \, [v_1, \ldots, v_t]$ has rank $t$, so that there are $t$ linearly independent rows. Denote by $M$ the $t \times t$ matrix formed from these rows, and by $N$ the $(K-t) \times t$ matrix consisting of the remaining rows, and assume after relabeling that 
\[ \Gamma \, [v_1, \ldots, v_t] = \left[\begin{array} {c} M \\ N \end{array} \right].\]
For $e_i \in \R^{K-t}$ the $i$th unit vector, we may set 
$ o^{(i)} = \big[e_i' N M^{-1}, - e_i' \big]'$. 
Now, if there exist $y \in \cS$ for which 
$(D_y\Gamma v_j)' o^{(i)} \neq 0
$ for some $i \in \{1 , \ldots , K-t\}$ and $j\in  \{ 1, \ldots ,t\}$, 
then $D_y\Gamma v_j$ cannot be contained in the $t$-dimensional subspace $\textnormal{span}\,  \{ \Gamma v_1, \ldots, \Gamma v_t\}$ of $\R^K$, and the assertion of the lemma follows. 

Thus assume that 
\begin{equation}
\label{eq:ygeneral}
(D_y\Gamma v_j)' o^{(i)} = 0 \quad (y \in \cS;\ i\in\{1,\hdots,K-t\}; \ j\in\{1,\hdots,t\}),
\end{equation}
this will lead to a contradiction. Let $\gamma_1, \ldots, \gamma_K$ denote the row vectors of $\Gamma$. Set 
\begin{equation*}
S_i = \textnormal{span}\, \big\{o^{(i)}_1 F_1(y) \gamma_1  + \cdots + o^{(i)}_t  F_t(y)  \gamma_t -   F_{t+i}(y)\gamma_{t+i} \; \mid \;  y \in \cS \big\} \quad (i=1, \ldots, K-t).
\end{equation*}
Then (\ref{eq:ygeneral}) implies that 
\begin{equation}
  	\label{eq:spans}
	\textnormal{span}\, \big\{S_1, \ldots, S_{K-t} \big\} \subseteq \textnormal{span}\, \big\{ v_1, \ldots, v_{t} \big\}^{\perp}. \\
\end{equation}
We first argue that if (\ref{eq:spans}) holds, 
\begin{equation}\label{eq:dimsizesi}
\dim S_i \geq 2 \quad (i =1, \ldots, K-t).
\end{equation}
To this end we assert that among the first $t$ elements of  $o^{(i)}$  there is at least one non-zero entry. Indeed, suppose that all $t$ entries were equal zero, then by the construction of $o^{(i)}$, definition of $S_i$ and (\ref{eq:spans}), we get that
$$
F_{t+i}(y)\gamma_{t+i} v_1 = 0, \qquad y \in \cS,
$$
a contradiction since $ \gamma_{t+i} v_1   > 0$ and since we assume that $v_1$ has strictly positive entries. 

Thus, assume that $j \in \{1, \ldots, t\}$ is such that $o^{(i)}_j \not=0$. Since $F_j$ and $F_{t+i}$ are distinct distribution functions, there exist $y^{(i)}_1, y_2^{(i)}$ such that the vectors
\[ \big[F_j(y^{(i)}_1), F_{t+i}(y^{(i)}_1)\big]', \qquad 
\big[F_j(y^{(i)}_2),  F_{t+i}(y^{(i)}_2)\big]'\] 
are linearly independent, and hence so are the vectors
\[  \big[ o^{(i)}_1 F_1(y^{(i)}_l), \ldots , o^{(i)}_t  F_t(y^{(i)}_l), -   F_{t+i}(y^{(i)}_l)\big]' \quad (l=1,2),\]
of coefficients of the linearly independent vectors $\gamma_1, \ldots, \gamma_t, \gamma_{t+i}$, which shows (\ref{eq:dimsizesi}).   
To conclude the proof, we observe that due to the linear independence of $\gamma_1, \ldots, \gamma_K$ and the definition of the $S_i$, we have that 
\[ S_i  \nsubseteq \textnormal{span} \Big\{ \bigcup_{j  =  1, \;  j\neq i }^{K-t} S_j \Big\}\quad  (i = 1, \ldots, K-t).\]
Together with (\ref{eq:dimsizesi}) we obtain  that 
\[ \dim \big(\textnormal{span}\,\big\{S_1, \ldots, S_{K-t} \big\}\big) \geq K-t +1,\]
a contradiction to (\ref{eq:spans}). This concludes the proof of the lemma. 
\end{proof}

\begin{lemma}\label{lemma:indone}
Under Assumptions \ref{eq:tpm} and \ref{eq:statedep}, for $T \geq K-1$ the conditional distributions of $W_T$ given $X_{T+1}=k$ ($k=1, \ldots, K$), that is the functions $G_T(\cdot ;k)$, are linearly independent, and furthermore, there exist $z_1, \ldots, z_K \in \cS^T$ such that the matrix 
\[ A_1 = \big\{G_T(z_t; k )\big\}_{k,t=1, \ldots, K}\]
has full rank $K$.
\end{lemma}

\begin{proof}
We show the claim for $T=K-1$. Since marginal distributions of linearly dependent distributions remain linearly dependent, linear independence then follows for any $T \geq K-1$, and the existence of corresponding points $z_1, \ldots, z_K \in \cS^T$ follows from Lemma 17 in \citet{allman}. Consider
\begin{align*}
 \tilde G_t(y_1^t; k) & = F_k(y_1)\, \sum_{k_2 = 1}^K \cdots \sum_{k_t = 1}^K \alpha_{k, k_2}\, \prod_{s=2}^{t-1} \alpha_{k_s, k_{s+1}}\, \prod_{s=2}^t F_{k_s} (y_s)\\
	& = F_k(y_1)\, \gamma_k D_{y_2}\, \Gamma \,\cdots \,   \Gamma D_{y_t} \, 1_K\quad (k=1, \ldots, K,\ t\in\{1,\hdots, K-1\}),
\end{align*}
where as above, $D_y = \textnormal{diag}\, \big\{F_1(y), \ldots, F_K(y)\big\}$ and $\gamma_k$ are the row vectors of $\Gamma$. 
Since
\begin{align*}
\tilde A_1 & =  \{\tilde G_{K-1}(z_t; k) \}_{k,t=1, \ldots, K} = \Gamma A_1,
\end{align*}
it is enough to show the claim of the lemma for $\tilde A_1$. 
We proceed by induction and show that  there exist 
\[ z_1^{(t)}, \ldots, z_{t+1}^{(t)} \in \cS^t\quad (t=1, \ldots, K-1),\]
for which the vectors
\[ v_j^{(t)} = \big[ \tilde G_t(z_j^{(t)}; 1 ), \ldots, \tilde G_t(z_j^{(t)}; K )\big]\quad (j=1, \ldots, t+1),\]
are linearly independent, and  $v_1^{(t)}$ has only strictly positive entries. The case $t=K-1$ will then establish Lemma \ref{lemma:indone}. 

Indeed, since the distribution functions are distinct, for $t=1$ we find $y_1^{(1)}, y_2^{(1)} \in \cS$ for which
\[ v_j^{(1)} = \big[F_1(y_j^{(1)}),\ldots, F_K(y_j^{(1)})  \big]'\quad (j=1,2),\]
are linearly independent, and for which $v_1^{(1)}$ has only positive entries.

Now, suppose that the claim is valid for $t \leq K-1$. We apply Lemma \ref{lemma:dimenextend} and find a $y_0 \in \cS$ and a $j \in \{1, \ldots, t+1\}$ for which the $K \times (t+2)$ matrix
\[ M = \big[\Gamma v_1^{(t)}, \ldots, \Gamma v_{t+1}^{(t)}, \ D_{y_0}\, \Gamma v_{j}^{(t)}\big] \]
has full rank $t+2$, which means that it has a $(t+2) \times (t+2)$ submatrix of non-zero determinant. 
Since $D_y \to I_K$, as $y \to \infty$,
\[ \big[ D_y \Gamma v_1^{(t)}, \ldots, D_y \Gamma v_{t+1}^{(t)}, \ D_{y_0}\, \Gamma v_{j}^{(t)}\big] \to M,\qquad y \to \infty,\]
and the corresponding submatrix will also be of non-zero determinant for an appropriate $y \in \cS$. The claim for $t+1$ now follows by setting
\[ z_s^{(t+1)} = \big[ y, (z_s^{(t)})\big] \quad (s=1, \ldots, t+1),  \qquad  z_{t+2}^{(t+1)} = \big[ y_0, (z_j^{(t)})\big],\]
which concludes the proof.
\end{proof}

%
\section{Proofs for nonparametric maximum likelihood estimation}
%
%
\subsection{Proofs of the main results}
Let $\mathcal D$ be a class of densities on $\cS$ with respect to some $\sigma$-finite measure $\nu$. 
Suppose that $(Y_t, X_t)$ is a $K$-state hidden Markov model with transition probability matrix $\Gamma_0$ satisfying Assumptions \ref{eq:tpm} and \ref{eq:stationary} and having stationary distribution $\pi_0$, and that the state-dependent distributions $F_{1,0}, \ldots, F_{K,0}$ are all distinct and have densities $f_{1,0}, \ldots ,f_{K,0}$  from the class $\mathcal D$. 

For parameters $\lambda$, $\Gamma$, $f_1, \ldots, f_K$, $n \in \N$ and $y_1^n \in \cS^n$ consider 
\[ g_n\big(y_1^n; \lambda, \Gamma, f_1, \ldots , f_K \big) = \sum \limits_{x_1 = 1}^{K} \cdots \sum \limits_{x_n = 1}^{K}  \lambda_{x_1} f_{x_1}(y_1) \prod \limits_{i=2}^n \alpha_{x_{i-1}, \; x_i} f_{x_i}(y_i),\]
the joint density of $n$ observations under these parameters, and denote the log-likelihood function of $Y_1, \ldots, Y_n$ by
\[ L_{n}\big(\lambda, \Gamma, f_1, \ldots, f_K \big) = \log g_n\big(Y_{1}^{n}; \lambda, \Gamma, f_1, \ldots, f_K \big).\]
\begin{assumption}\label{eq:intble} 
The true densities $f_{j,0}\in \mathcal{D}$ satisfy $E \{|\log f_{j,0}(Y_1)|\} < \infty$,\quad ($j=1,\hdots,K$); 
\end{assumption}

\begin{assumption}\label{eq:intpospart} 
The model satisfies $E \{\log f (Y_1) \}^+ < \infty$, \quad ($f \in \mathcal{D}$). 
\end{assumption}

\begin{theorem}\label{the:mlcontrast}
Suppose that $(Y_t, X_t)$ is a $K$-state hidden Markov model with transition probability matrix $\Gamma_0$ satisfying Assumptions \ref{eq:tpm} and \ref{eq:stationary}, and that the state-dependent distributions $F_{1,0}, \ldots, F_{K,0}$ are all distinct and have densities $f_{1,0}, \ldots ,f_{K,0}$  from the class $\mathcal D$, and satisfy Assumption \ref{eq:intble}.  Let $\lambda, \lambda_0$ be $K$-state probability vectors with strictly positive entries. 
Under Assumption \ref{eq:intpospart}, given $f_1, \ldots, f_K \in \mathcal D$ we have almost surely as $n \to \infty$ that
\begin{align*}
\begin{split}\label{eq:kullbackhmm}
 n^{-1} \big\{L_{n}\big(\lambda, \Gamma, f_1, \ldots, f_K \big) & - L_{n}\big(\lambda_0,\Gamma_0, f_{1,0}, \ldots, f_{K,0}  \big) \big\}\\
	& \to - K\{(\Gamma_0, f_{1,0}, \ldots, f_{K,0}),(\Gamma, f_1, \ldots, f_K)\} \in (- \infty, 0],
\end{split}
\end{align*}
and $K\{(\Gamma_0, f_{1,0}, \ldots, f_{K,0}),(\Gamma, f_1, \ldots, f_K) \} = 0$ if and only if the two sets of parameters are equal up to label swapping. 
\end{theorem}

\begin{proof}
The existence of the limit and its independence from the starting distributions may be deduced from Kingman's subadditive ergodic theorem as shown in \citet{ler}. To show definiteness, we briefly recall a construction from \citet{ler}. For a sequence $(y_n)$ in $\cS$,
define sequences 
\[ u^{(n)}, v^{(n)} \in {\Delta}^{K-1} = \big\{(s_1, \ldots, s_K)' \in [0,1]^K:\quad s_1 + \ldots + s_K=1 \big\},\]
by
\begin{equation*}
	\begin{aligned}
	u_k^{(1)} &= \pi_{0k}, \quad u_k^{(n+1)} = \frac{\sum_{j=1}^K u_j^{n} f_{0j}(y_n) \alpha_{0,jk} }{\sum_{j=1}^K u_j^{n} f_{0j}(y_n)} \quad (k=1,\ldots,K; \;\; n=1,2\ldots)\\
	v_k^{(1)} &= \pi_{0k}, \quad v_k^{(n+1)} = \frac{\sum_{j=1}^K v_j^{n} f_j(y_n) \alpha_{j,k} }{\sum_{j=1}^K v_j^{n} f_j(y_n)} \quad (k=1,\ldots,K; \;\; n=1,2\ldots),
	\end{aligned}
\end{equation*}
where $\pi_0$ is the stationary distribution of $\Gamma_0$, and we set $0/0=0$.  Let $\Omega = \{(y_n, u^{(n)}, v^{(n)})_{n\in\N}\}$. \citet{ler} shows that there is a probability measure on $\Omega$, such that if $Q(u,v)$ denotes the distribution of $\big(u^{(1)}, v^{(1)}\big)$ under this measure, for any $T \in \N$ we have that
\begin{align*}
 T\, K\{(\Gamma_0, f_{1,0}, \ldots, f_{K,0}),(\Gamma, f_1, \ldots, f_K) \}
=  \int \limits_{\Delta^{K-1} \times {\Delta}^{K-1}} \int\limits_{\cS^T} g_T(y_1^T;u, \Gamma_0, f_{1,0}, \ldots, f_{K,0})&\\
	  \cdot\log \Big\{\frac{g_T(y_1^T;u, \Gamma_0, f_{1,0}, \ldots, f_{K,0})}{g_T(y_1^T;v, \Gamma, f_{1}, \ldots, f_{K})} \Big\}\, d\nu^{\otimes T}(y_1^T)\, dQ(u,v).&
\end{align*}
Since the inner integral is an ordinary Kullback--Leibler divergence of two densities, non-negativity is obvious. To show definiteness, choose $T = (2 K+1) (K^2-2K+2) + 1$. Suppose that the two sets of parameters $\Gamma_0, f_{1,0}, \ldots, f_{K,0}$ and $\Gamma, f_1, \ldots, f_K  $ are not equal up to label swapping. Then from Theorem \ref{the:identgeneral}, for any $u,v \in \Delta^{K-1}$, this Kullback--Leibler divergence
\[ \int\limits_{\cS^T} g_T(y_1^T;u, \Gamma_0, f_{1,0}, \ldots, f_{K,0})
	\log \Big\{\frac{g_T(y_1^T;u, \Gamma_0, f_{1,0}, \ldots, f_{K,0})}{g_T(y_1^T;v, \Gamma, f_{1}, \ldots, f_{K})} \Big\}\, d\nu^{\otimes T}(y_1^T) >0,\]
which implies definiteness. 
\end{proof}
%
%
Suppose that $(f_{\vartheta})_{\vartheta\in \Theta}$ is a parametric family of densities on $\cS$ with respect to some $\sigma$-finite measure, where $\Theta \subset \R^d$ is compact. Let $\tilde{\Theta}$ be the set of Borel probability measures on $\Theta$, endowed with the weak topology it is also a compact set. Assume that the map $(y,\vartheta) \mapsto f_{\vartheta}(y)$ is continuous on $\cS\times \Theta$. Given $\mu \in \tilde{\Theta}$, we let 
\[ f_\mu(y) = \int_\Theta f_{\vartheta}(y)\, d \mu(\vartheta)\]
denote the corresponding mixture density. Let 
$ \theta = \big(\Gamma, \mu_1, \ldots, \mu_K) \in G \times \tilde{\Theta} \times  \cdots \times \tilde{\Theta}$, 
where $G$ is the compact set of $K$-state transition probability matrices. Given the sequence of observations $Y_1, \ldots, Y_n$, the log-likelihood function is 
\[ L_n(\theta) = \log\, \big\{ \sum \limits_{x_1 = 1}^{K} \cdots \sum \limits_{x_n = 1}^{K}  \lambda_{x_1} f_{\mu_{x_1}}(Y_1) \prod \limits_{i=2}^n \alpha_{x_{i-1}, \; x_i} f_{\mu_{x_i}}(Y_i)\, \big\},\]
where $\lambda$ is an arbitrary $K$-state strictly positive probability vector.
\begin{theorem}\label{prop:existencemle}
Let $(f_\vartheta)_{\vartheta\in\Theta}$ be a parametric family of densities with $\Theta\subset\mathbb{R}^d$ compact, and let $\mathcal{D}=\{f_\mu:~\mu\in\tilde{\Theta}\}$ be the model for the state-dependent densities, where $\tilde{\Theta}$ is the set of Borel probability measures on $\Theta$. Then, for any $n \geq 1$ there exists a maximum likelihood estimator $\hat \theta_n = (\hat \Gamma_n,\hat \mu_{1,n},\ldots,\hat \mu_{K,n})$, for which the state-dependent mixing distributions are of the form
\[ \hat \mu_{k,n} = \sum_{j=1}^m a_j\, \delta_{\vartheta_{j,k}} \quad (k=1, \ldots, K),\]
where $m\in \{1,\hdots,Kn+1\}$, $a_j>0$, $\sum_{j=1}^m a_j=1$, 
$\vartheta_{j,k}\in \Theta$ ($j=1, \ldots, m$) and where $\delta_\vartheta$ is the point-mass at $\vartheta$. 
\end{theorem}

\begin{proof}
Since by assumption, $\Theta$ is compact and $f_\cdot(y)$ is continuous for any $y \in \cS$, if $\mu_n \to \mu$ weakly then $\int_{\Theta} f_\vartheta(y)\mu_n(d\vartheta) \rightarrow \int_\Theta f_\vartheta(y)\mu(d\vartheta)$. 
Therefore, the map 
\begin{align*}
\Psi: \tilde{\Theta}\times \cdots \times \tilde{\Theta} &\longrightarrow  \mathbb{R}^n \times \cdots \times \mathbb{R}^n \\
(\mu_1,\hdots,\mu_K)& \longmapsto \big[\{f_{\mu_1}(y_t)\}_{t=1, \ldots, n},\hdots,\{f_{\mu_K}(y_t)\}_{t=1, \ldots, n}\big],
\end{align*}
is affine and continuous. Since $\tilde \Theta$ is compact as well, the image 
\begin{equation*}
	D = \Psi( \tilde{\Theta}\times \cdots \times \tilde{\Theta}) \subset \mathbb{R}^{Kn}
\end{equation*}
is compact and convex. For fixed $\Gamma = \big(\alpha_{j,k}\big)_{j,k=1, \ldots, K}$, we need to consider maximizing 
\[ \tilde L_n\big(t_1, \ldots, t_K \big) = \log\, \Big(\sum_{x_1=1}^K\cdots\sum_{x_n=1}^K \lambda_{x_1}t_{x_1,1} \,\prod_{j=2}^n \alpha_{x_{j-1} x_j} t_{x_j,j}\Big)\]
over $D$, where $t_k = (t_{k,1}, \ldots, t_{k,n})$ and $z= [t_1, \ldots, t_K] \in D$. 
Since $D$ is compact and $\tilde L_n$ is continuous, there exists $z^*=(t_1^*,\hdots,t_K^*) \in D$, $t_i^* \in \mathbb{R}^n$ ($i=1,\hdots,K$) where $\tilde L_n$ is maximal. Since $D$ is also convex, according to Carath\'eodory's theorem $z^*$ can be expressed by a convex combination of at most $Kn+1$ extreme points $s_j^* \in D$, so that 
\begin{equation}\label{caratheodory}
z^*= \sum_{j=1}^{Kn+1} a_j s_j^*,
\end{equation}
where the weights $a_j\geq 0$ sum to one. The $s_j^*$ are images of extreme points in $\tilde{\Theta}\times\cdots\times\tilde{\Theta}$ under the affine map $\Psi$ \citep[see][]{Simon2011}. Further, points in the Cartesian product $\tilde{\Theta}\times\cdots\times\tilde{\Theta}$ are extreme if and only if all coordinates are extreme in $\tilde \Theta$, and the extreme points in  $\tilde \Theta$ are given by the point masses $\delta_\vartheta$ for a $\vartheta \in \Theta$. Therefore, there exist $\vartheta_{k,j} \in \Theta$ ($k=1, \ldots, K$; \ $j=1, \ldots, nK+1$) such that 
\[ \Psi\big(\delta_{\vartheta_{1,j}}, \ldots , \delta_{\vartheta_{K,j}} \big) = s_j^*.\]
Let $m \in \{1,\hdots,Kn+1\}$ denote the number of extreme points needed in the convex combination \eqref{caratheodory} for which $a_j>0$. Then, after relabeling we obtain
\begin{align*}
z^*&= \sum_{j=1}^m a_j s_j^*
	  = \sum_{j=1}^m a_j \Psi\big(\delta_{\vartheta_{1,j}}, \ldots , \delta_{\vartheta_{K,j}} \big) 
    = \Psi\Big(\sum_{j=1}^m a_j \delta_{\vartheta_{1,j}}, \ldots, \sum_{j=1}^m a_j \delta_{\vartheta_{K,j}}\Big).
\end{align*}
Since 
\[
\sup_{(\Gamma,\mu_1,\hdots,\mu_K)}L_n(\theta) = \sup_{\Gamma} \sup_{\mu_1,\hdots,\mu_K} L_n(\theta)
\]
the theorem follows. 
\end{proof}
%
%
%
Assume that the true state-dependent densities $f_{k,0} = f_{\mu_{k,0}}$ belong to the model and are all distinct, and that $\Gamma_0 \in G$ satisfies Assumption \ref{eq:tpm}. 
\begin{assumption} \label{ass_finiteexp}
For every $\mu \in \tilde{\Theta}$ and a small enough neighborhood $O_{\mu}$ of $\mu$ we have
\begin{equation*}
E \big[\sup_{\tilde \mu\in O_\mu}\{\log f_{\tilde \mu}(Y_1)\}^{+}\big] < \infty.
\end{equation*}
\end{assumption}
%
%
\begin{theorem}\label{th:consistency}
Let $(f_{\vartheta})_{\vartheta\in \Theta}$ be a parametric family of densities with $\Theta \subset \R^d$ compact, and let $\mathcal D = \{f_\mu:\, \mu \in \tilde \Theta\}$ be the model for the state-dependent densities, where $\tilde{\Theta}$ is the set of Borel probability measures on $\Theta$.
Suppose that Assumptions \ref{eq:tpm}, \ref{eq:stationary}, \ref{eq:intble} and \ref{ass_finiteexp} hold and let  $\hat \theta_n = (\hat \Gamma_n,\hat \mu_{1,n},\ldots,\hat \mu_{K,n})$ denote a maximum likelihood estimator. Then, after relabeling, we have in probability as $n \to \infty$ that $\hat \Gamma_n \to \Gamma_0$ and for any $y \in \cS$ and $k \in \{1, \ldots, K\}$ that
\[ f_{\hat \mu_{k,n}}(y) \to f_{k,0}(y).\]
Furthermore, if the mixing distribution $\mu$ is identified from the mixture density $f_\mu$, then we have additionally that $d_w\big(\hat \mu_{k,n},\mu_{k,0}\big) \to 0$ in probability, where $d_w$ is a distance which metrizes weak convergence in $\tilde \Theta$.  
\end{theorem}
\begin{proof}
We let 
\[ \tilde \Theta_{k,0} = \big\{ \mu \in \tilde \Theta:\quad f_\mu = f_{\mu_{k,0}} \big\} \quad (k=1, \ldots, K).\]
Moreover, for the proof we set
\[ \Lambda = G \times \tilde{\Theta}\times \cdots \times \tilde{\Theta} \qquad \text{and} \qquad \Lambda_0 = \{\Gamma_0\} \times \tilde \Theta_{1,0} \times \cdots \times \tilde \Theta_{K,0}.\]
We shall metrize weak convergence on $\tilde \Theta$ using the bounded Lipschitz metric \citep{vandervaartwellner1996}
\begin{equation*}
 d_{BL}(\mu_1,\mu_2)= \sup\Big\{\big|\int f d\mu_1 - \int f d\mu_2\big|;~ f:\Theta \longrightarrow [0,1],~|f(\vartheta_1)-f(\vartheta_2)|\leq d(\vartheta_1,\vartheta_2)\Big\}.
\end{equation*}
On $G$ we take any metric equivalent to the Euclidean metric and on $\Lambda$, we take the product metric which we denote by $d$ . The proof consists of two steps.
\setcounter{step}{0}
\begin{step}\label{step1consistency}
Show that in probability as $n \to \infty$, 
\[ d\big(\hat \theta_n, \Lambda_0\big) \to 0,\]
which in particular implies $\hat \Gamma_n \to \Gamma_0$.
\end{step}
\begin{step}\label{step2consistency}
Show that from the convergence in probability
\[ d_{BL}\big(\hat \mu_{n,k}, \tilde \Theta_{k,0} \big) \to 0\]
it follows that for any $y \in \cS$, in probability
\[ f_{\hat \mu_{n,k}}(y) \to f_{k,0}(y).\] 
\end{step}
Consider first step \ref{step2consistency}. Since $\Theta$ is compact, the function $\vartheta \mapsto f_\vartheta(y)=g(\vartheta)$ is uniformly continuous and bounded. Therefore, given $\varepsilon$, from Lemma \ref{lemma:lipschitzapprox} below there is a Lipschitz-continous $h$ such that $|g(\vartheta) - h(\vartheta)| < \varepsilon$. Let $K_1(\varepsilon)= \sup_{\vartheta\in\Theta}|h(\vartheta)|$ and let  $K_2(\varepsilon)$ denote the Lipschitz-constant of $h$, and let $K(\varepsilon)=\max\{K_1(\varepsilon),K_2(\varepsilon)\}$. 

Given any $\mu \in \tilde \Theta$, there is a $\nu \in \tilde \Theta_{k,0}$ for which
$d_{BL}\big(\mu, \nu\big) \leq d_{BL}\big( \mu, \tilde \Theta_{k,0} \big) + \varepsilon/K(\varepsilon)$. From the definition of $\tilde \Theta_{k,0}$, 
\[ f_{k,0}(y) = \int g(\vartheta)d \nu(\vartheta).\]
Therefore, we may estimate
\begin{align*}
\big| f_{\mu}(y) -f_{k,0}(y)\big| &= \Big|\int g(\vartheta) d\mu(\vartheta)- \int g(\vartheta)d \nu(\vartheta)\Big| \\
&\leq \int \big|g(\vartheta)-h(\vartheta)\big|\, d\mu(\vartheta) + \Big|\int h(\vartheta)d \mu(\vartheta)  - \int h(\vartheta) d \nu(\vartheta)\Big| \\&\quad + \int \big| h(\vartheta) - g(\vartheta) \big|\, d\nu(\vartheta) \\
&\leq 2\, \varepsilon + K(\varepsilon)\, d_{BL}\big(\mu, \nu\big) 
\leq 3\, \varepsilon + K(\varepsilon)\, d_{BL}\big(\mu, \tilde \Theta_{k,0}\big). 
\end{align*}
Letting $\delta= \varepsilon/K(\varepsilon)$, we therefore obtain
\[ \text{pr}\,\{\big| f_{\hat \mu_{n,k}}(y) -f_{k,0}(y)\big| > 4 \varepsilon \} \leq \text{pr}\,\{ d_{BL}\big(\hat \mu_{n,k}, \tilde \Theta_{k,0}\big) > \delta \} \to 0.\]

For step \ref{step1consistency}.~we may follow the argument in \citet{ler}. For a parameter vector $\theta = \big(\Gamma, \mu_1, \ldots, \mu_K)$ set  
\begin{align*}
M_{s,t}(\theta)& = \max_{k\in\{1,\hdots,K\}} \, \big\{f_{\mu_{k}}(Y_{s+1}) \,  \sum \limits_{x_2 = 1}^{K} \cdots \sum \limits_{x_{t-s} = 1}^{K}  \, \alpha_{k,x_2}\, f_{\mu_{x_2}}(Y_{s+2})\,\\
& \qquad \quad  \prod \limits_{i=3}^{t-s} \alpha_{x_{i-1}, \; x_i} f_{\mu_{x_i}}(Y_{s+i})\, \big\} \quad (s,t \in \N_0; \ s < t). 
\end{align*}
From \citet[Lemma 3]{ler} we have 
\[ M_{s,t}(\theta)\leq M_{s,u}(\theta)\, M_{u,t}(\theta) \quad (s < u < t),\]
so that the process $\log M_{s,t}(\theta)$ is subadditive. From Kingmans' subadditive ergodic theorem, 
\[ n^{-1}\, \log\, M_{0,n}(\theta) \to H(\theta_0, \theta)\]
in $L_1$ and almost surely, and from the arguments in \citet{ler} and Theorem \ref{the:mlcontrast},  
\begin{align*}
 H(\theta_0, \theta) & \leq H(\theta_0, \theta_0), \qquad 
H(\theta_0, \theta)  = H(\theta_0, \theta_0) \ \Leftrightarrow \ \theta \in \Lambda_0.
\end{align*}
Similarly, for a $\theta \in \Theta$ and a neighborhood $O_\theta$ of $\theta$, the process 
$\log \sup_{\theta \in O_\theta} M_{s,t}(\theta)$ is subadditive as well, and the limit
\[ n^{-1} \, \log\, \sup_{\theta \in O_\theta} M_{0,n}(\theta) \to H(\theta_0, \theta, O_\theta).\]
The argument in \citet{ler} shows that there exists a  $\delta>0$, such that for every $\theta \in \Lambda \setminus \Lambda_0$ there is a neighborhood $O_\theta$ for which
\[ H(\theta_0, \theta, O_\theta) \leq H(\theta_0, \theta_0) - \delta.\]
Given $\varepsilon>0$ let $\Lambda_\varepsilon=\{\lambda\in \Lambda, ~d(\lambda,\Lambda_0)\geq \varepsilon\}$. Since $\Lambda$ is compact and the distance is continuous, $\Lambda_\varepsilon$ is compact. Therefore, we can find finitely many $O_{\theta_j}$ ($j=1, \ldots, m$), which cover $\Lambda_\varepsilon$. 
Since $ L_n(\theta)$ and $\log M_{0,n}(\theta)$ have the same asymptotics, we obtain
\begin{equation*}
n^{-1}\sup_{\theta\in \Lambda_\varepsilon} L_n(\theta) \leq n^{-1}\max_{j=1,\hdots,m}  \sup_{\theta\in O_{\theta_j}} L_n(\theta) \longrightarrow \max_{j=1,\hdots,m} H(\theta_0, \theta_j, O_{\theta_j}) \leq H(\theta_0,\theta_0)- \delta.
\end{equation*}
Since $n^{-1}\, L_n(\theta_0) \rightarrow H(\theta_0,\theta_0)$, and
\[ \big\{\hat \theta_n \in \Lambda_\varepsilon \big\} \subset \Big\{n^{-1}\sup_{\theta\in \Lambda_\varepsilon} L_n(\theta) \geq n^{-1}\, L_n(\theta_0)\Big\},\]
we obtain $\text{pr}\,\big(\hat \theta_n \in \Lambda_\varepsilon\big) \to 0$.
\end{proof}

\subsection{Proof of an additional lemma}

In the proof of Theorem \ref{th:consistency}, we used the following well-known lemma, for which, using argument in \citet{Garrido2008}, we provide a proof for convenience of the reader. 
\begin{lemma} \label{lemma:lipschitzapprox}
Let $(\Theta,d)$ be an arbitrary metric space, not necessarily compact. Every bounded and uniformly continuous function $g$ on $\Theta$ can be uniformly approximated by Lipschitz-continuous functions. 
\end{lemma}
\begin{proof}
If $g$ is bounded and uniformly continuous, then so are its positive and negative part. Therefore, we may assume that the given function $g \geq 0$. 
Choose $M>0$ such that $|g(\vartheta)|<M$ for all $\vartheta\in \Theta$, and given $\varepsilon>0$ let $N \in \N$ such that $(N+1)\varepsilon \geq M$. Define the sets 
\begin{equation*}
C_n=\{\vartheta \in \Theta:~(n-1)\varepsilon < g(\vartheta) < (n+1)\varepsilon\} \quad (n=0,1,\hdots,N), 
\end{equation*}
which cover $\Theta$. Evidently, $C_n \cap C_m = \emptyset$ for $|n-m|>1$. Using the uniform continuity of $g$, we may choose $\delta>0$ such that $|g(\eta) - g(\vartheta)| < \varepsilon/2$ whenever $d(\eta, \vartheta)< \delta$. 
First we show that for every $\vartheta \in \Theta$ there is a $m\in \{0,\hdots,N\}$ with 
\begin{equation}\label{eq:formerclaim}
 B_\delta(\vartheta)=\{\eta \in \Theta:~d(\vartheta,\eta)<\delta\}\subseteq C_m.
\end{equation}
For the proof, first observe that if $\vartheta$ is contained in just a single set $C_m$, we must have $g(\vartheta) = m \varepsilon$, and $B_\delta(\vartheta)\subseteq C_m$ is obvious by the choice of $\delta$ and the definition of $C_m$. Now suppose that $\vartheta \in C_n \cap C_{n+1}$ for some $n \in \{0, \ldots, N-1\}$, so that $n \varepsilon < g(\vartheta) < (n+1) \varepsilon$. If $n \varepsilon < g(\vartheta) \leq (n+1/2) \varepsilon$ we take  $m=n$, otherwise we take $m=n+1$, then (\ref{eq:formerclaim}) follows. 

Now define the functions
\begin{equation*}
g_n(\vartheta)= \inf\{1, d(\vartheta, \Theta \backslash C_n)\} \in [0,1], 
\end{equation*}
where $d(\vartheta, \emptyset)=\infty$. The $g_n$ are supported on $C_n$ and Lipschitz continuous with constant $1$, since for $\vartheta_1 \not= \vartheta_2$ 
\begin{equation*}
\frac{|g_n(\vartheta_1)-g_n(\vartheta_2)|}{d(\vartheta_1,\vartheta_2)}\leq \frac{|d(\vartheta_1,\Theta \backslash C_n)-d(\vartheta_2,\Theta\backslash C_n)|}{d(\vartheta_1,\vartheta_2)}\leq \frac{d(\vartheta_1,\vartheta_2)}{d(\vartheta_1,\vartheta_2)}=1.
\end{equation*} 
Define $h(\vartheta)= \sum_{n=0}^N g_n(\vartheta)$. From (\ref{eq:formerclaim}), we have that $h(\vartheta) \geq \delta$ for all $\vartheta \in \Theta$, and since each $\vartheta$ can at most be contained in two sets $C_n, C_{n+1}$, we have $h(\vartheta) \leq 2$. 
Further, for $\vartheta_1,\vartheta_2 \in \Theta$ we have
\begin{equation*}
|h(\vartheta_1)-h(\vartheta_2)|\leq \sum_{n=0}^N |g_n(\vartheta_1)-g_n(\vartheta_2)| \leq (N+1)d(\vartheta_1,\vartheta_2),
\end{equation*}
which proves that $h$ is a Lipschitz function with constant $(N+1)$.
Now, set $\tilde{h}(\vartheta)=h(\vartheta)^{-1}\sum_{n=0}^N n g_n(\vartheta)$. We shall show that $\tilde{h}$ is also Lipschitz continuous and that 
\begin{equation}\label{eq:approxprop}
 \sup_{\vartheta \in \Theta} \big|g(\vartheta) - \varepsilon \, \tilde h(\vartheta) \big| \leq 2\, \varepsilon.
\end{equation}
To this end, we compute
\begin{align*}\label{eq:consthtilde}
\begin{split}
|\tilde{h}(\vartheta_1)-\tilde{h}(\vartheta_2)|
&\leq \sum_{n=0}^N \Big|\frac{1}{h(\vartheta_1)}n g_n(\vartheta_1) - \frac{1}{h(\vartheta_2)}n g_n(\vartheta_2)\Big| \\
& \leq \sum_{n=0}^N \frac{ng_n(\vartheta_1) \big|h(\vartheta_2) - h(\vartheta_1) \big|}{h(\vartheta_1)h(\vartheta_2)} + \sum_{n=0}^N \frac{n \big|g_n(\vartheta_1) -g_n(\vartheta_2)\big|}{h(\vartheta_2)} \\
&\leq \big\{(N+1)^3 \, \delta^{-2}\,  + (N+1)^2 \delta^{-1}\big\}\, d(\vartheta_1, \vartheta_2),
\end{split}
\end{align*}
by the properties of $h$ and the $g_n$. As for (\ref{eq:approxprop}), suppose that 
$\vartheta\in C_m$. Then
\begin{align*}
|\varepsilon\, \tilde{h}(\vartheta)-g(\vartheta)|&\leq  \varepsilon\, \Big|\frac{(m-1)g_{m-1}(\vartheta)+mg_m(\vartheta)+(m+1)g_{m+1}(\vartheta)}{g_{m+1}(\vartheta)+g_m(\vartheta)+g_{m-1}(\vartheta)}-m\Big| + \big| \varepsilon\, m - g(\vartheta)\big|\\
& \leq \varepsilon\, \Big|\frac{g_{m-1}(\vartheta)-g_{m+1}(\vartheta)}{g_{m+1}(\vartheta)+g_m(\vartheta)+g_{m-1}(\vartheta)}\Big| + \varepsilon 
\leq 2 \varepsilon.
\end{align*}
\end{proof}
%
%
\section{Additional simulation results for different numbers of observations}
For the simulation scenario considered in section $4\cdot 1$ we report additional simulation results for several choices of the sample size $n$. Tables \ref{tab:relerr250}--\ref{tab:relerr5000} illustrate the consistency of the nonparametric maximum likelihood estimator. Further, one observes that for series shorter than 1000, the nonparametric maximum likelihood estimator is not superior to the misspecified parametric models in terms of the relative errors. 

\begin{table}[h]
\centering
\begin{tabular}{rrrrrrrrrr}
 $y$ & $-$15$\cdot$45 & $-$13$\cdot$77 & $-$11$\cdot$22 & $-$9$\cdot$05 & $-$7$\cdot$26 & $-$5$\cdot$3 & $-$2$\cdot$86 & $-$0$\cdot$21 & 1$\cdot$56 \\ 
nonpar & 143$\cdot$96 & 40$\cdot$54 & 12$\cdot$65 & 12$\cdot$54 & 22$\cdot$98 & 8$\cdot$22 & 42$\cdot$53 & 46$\cdot$63 & 52$\cdot$72 \\ 
2-comp & 148$\cdot$17 & 40$\cdot$93 & 11$\cdot$31 & 12$\cdot$19 & 24$\cdot$31 & 8$\cdot$15 & 43$\cdot$72 & 47$\cdot$76 & 55$\cdot$79 \\ 
Gauss & 162$\cdot$31 & 40$\cdot$59 & 8$\cdot$88 & 10$\cdot$75 & 22$\cdot$67 & 5$\cdot$99 & 41$\cdot$86 & 49$\cdot$97 & 51$\cdot$73 \\ 
\\
$y$    & $-$9$\cdot$36 & $-$6$\cdot$36 & $-$2$\cdot$71 & $-$0$\cdot$68 & 0$\cdot$5 & 1$\cdot$67 & 3$\cdot$71 & 7$\cdot$36 & 10$\cdot$36 \\ 
nonpar & 66$\cdot$14 & 34$\cdot$30 & 65$\cdot$12 & 12$\cdot$18 & 17$\cdot$28 & 21$\cdot$57 & 26$\cdot$74 & 68$\cdot$16 & 86$\cdot$82 \\ 
2-comp & 67$\cdot$78 & 35$\cdot$72 & 68$\cdot$21 & 11$\cdot$64 & 15$\cdot$90 & 20$\cdot$49 & 26$\cdot$90 & 69$\cdot$46 & 88$\cdot$15 \\ 
Gauss & 74$\cdot$89 & 32$\cdot$72 & 76$\cdot$71 & 9$\cdot$89 & 16$\cdot$80 & 21$\cdot$63 & 27$\cdot$35 & 77$\cdot$79 & 95$\cdot$70 \\ 
\\
$y$    & 2$\cdot$27 & 3$\cdot$74 & 6 & 7$\cdot$99 & 9$\cdot$66 & 11$\cdot$61 & 14$\cdot$93 & 20$\cdot$17 & 22 \\ 
nonpar & 1069$\cdot$26 & 157$\cdot$78 & 13$\cdot$36 & 20$\cdot$92 & 15$\cdot$64 & 10$\cdot$49 & 14$\cdot$01 & 39$\cdot$37 & 54$\cdot$23 \\ 
2-comp & 1090$\cdot$58 & 165$\cdot$61 & 11$\cdot$63 & 21$\cdot$49 & 14$\cdot$58 & 9$\cdot$57 & 14$\cdot$08 & 41$\cdot$38 & 55$\cdot$06 \\ 
Gauss & 1280$\cdot$50 & 207$\cdot$43 & 5$\cdot$60 & 25$\cdot$13 & 20$\cdot$07 & 8$\cdot$00 & 5$\cdot$33 & 35$\cdot$27 & 50$\cdot$02 \\ 
\end{tabular}
\caption{Relative errors ($\times100$) of the three estimators compared to the true densities at selected values for $y$ averaged over 10000 series of length 250. `Gauss' stands for Gaussian state-dependent distributions, `2-comp' for two component Gaussian mixtures and `nonpar' for nonparametric Gaussian mixtures. }
\label{tab:relerr250}
\end{table}
\begin{table}[b]
\centering
\begin{tabular}{rrrrrrrrrr}
$y$ & $-$15$\cdot$45 & $-$13$\cdot$77 & $-$11$\cdot$22 & $-$9$\cdot$05 & $-$7$\cdot$26 & $-$5$\cdot$3 & $-$2$\cdot$86 & $-$0$\cdot$21 & 1$\cdot$56 \\ 
nonpar & 125$\cdot$44 & 33$\cdot$44 & 9$\cdot$30 & 12$\cdot$27 & 23$\cdot$50 & 6$\cdot$31 & 43$\cdot$16 & 46$\cdot$72 & 45$\cdot$58 \\ 
2-comp & 130$\cdot$18 & 33$\cdot$87 & 8$\cdot$43 & 11$\cdot$98 & 24$\cdot$81 & 6$\cdot$28 & 44$\cdot$36 & 48$\cdot$08 & 49$\cdot$25 \\ 
Gauss & 146$\cdot$53 & 34$\cdot$89 & 7$\cdot$23 & 10$\cdot$50 & 23$\cdot$66 & 5$\cdot$24 & 42$\cdot$57 & 51$\cdot$15 & 41$\cdot$82 \\ 
\\
$y$    & $-$9$\cdot$36 & $-$6$\cdot$36 & $-$2$\cdot$71 & $-$0$\cdot$68 & 0$\cdot$5 & 1$\cdot$67 & 3$\cdot$71 & 7$\cdot$36 & 10$\cdot$36 \\ 
nonpar & 64$\cdot$66 & 27$\cdot$17 & 64$\cdot$75 & 10$\cdot$53 & 14$\cdot$94 & 20$\cdot$04 & 25$\cdot$38 & 64$\cdot$05 & 78$\cdot$08 \\ 
2-comp & 67$\cdot$96 & 28$\cdot$15 & 68$\cdot$57 & 10$\cdot$67 & 14$\cdot$55 & 19$\cdot$63 & 25$\cdot$58 & 65$\cdot$51 & 78$\cdot$84 \\ 
Gauss & 76$\cdot$33 & 23$\cdot$78 & 75$\cdot$83 & 9$\cdot$58 & 15$\cdot$77 & 20$\cdot$73 & 26$\cdot$07 & 80$\cdot$07 & 97$\cdot$33 \\ 
\\
$y$    & 2$\cdot$27 & 3$\cdot$74 & 6 & 7$\cdot$99 & 9$\cdot$66 & 11$\cdot$61 & 14$\cdot$93 & 20$\cdot$17 & 22 \\ 
nonpar & 1075$\cdot$62 & 161$\cdot$61 & 11$\cdot$29 & 20$\cdot$00 & 14$\cdot$12 & 7$\cdot$96 & 9$\cdot$78 & 35$\cdot$44 & 50$\cdot$19 \\ 
2-comp & 1093$\cdot$36 & 170$\cdot$46 & 9$\cdot$72 & 21$\cdot$91 & 14$\cdot$38 & 7$\cdot$26 & 9$\cdot$90 & 38$\cdot$64 & 51$\cdot$30 \\ 
Gauss & 1255$\cdot$34 & 204$\cdot$92 & 5$\cdot$11 & 24$\cdot$59 & 19$\cdot$35 & 7$\cdot$24 & 4$\cdot$16 & 34$\cdot$04 & 50$\cdot$41 \\ 
\end{tabular}
\caption{Relative errors ($\times100$) of the three estimators compared to the true densities at selected values for $y$ averaged over 10000 series of length 500. `Gauss' stands for Gaussian state-dependent distributions, `2-comp' for two component Gaussian mixtures and `nonpar' for nonparametric Gaussian mixtures. }
\label{tab:relerr500}
\end{table}
\begin{table}[t]
\centering
\begin{tabular}{rrrrrrrrrr}
$y$& $-$15$\cdot$45 & $-$13$\cdot$77 & $-$11$\cdot$22 & $-$9$\cdot$05 & $-$7$\cdot$26 & $-$5$\cdot$3 & $-$2$\cdot$86 & $-$0$\cdot$21 & 1$\cdot$56 \\ 
nonpar & 116$\cdot$05 & 30$\cdot$24 & 7$\cdot$85 & 12$\cdot$60 & 23$\cdot$79 & 5$\cdot$55 & 43$\cdot$60 & 46$\cdot$83 & 44$\cdot$15 \\ 
2-comp & 122$\cdot$20 & 30$\cdot$72 & 7$\cdot$08 & 12$\cdot$25 & 25$\cdot$06 & 5$\cdot$44 & 44$\cdot$81 & 48$\cdot$33 & 48$\cdot$88 \\ 
Gauss & 140$\cdot$12 & 32$\cdot$65 & 6$\cdot$40 & 10$\cdot$60 & 24$\cdot$13 & 4$\cdot$89 & 42$\cdot$98 & 52$\cdot$30 & 38$\cdot$85 \\ 
  \\
$y$    & $-$9$\cdot$36 & $-$6$\cdot$36 & $-$2$\cdot$71 & $-$0$\cdot$68 & 0$\cdot$5 & 1$\cdot$67 & 3$\cdot$71 & 7$\cdot$36 & 10$\cdot$36 \\ 
nonpar & 65$\cdot$00 & 23$\cdot$96 & 64$\cdot$76 & 9$\cdot$97 & 13$\cdot$96 & 19$\cdot$50 & 25$\cdot$26 & 61$\cdot$50 & 72$\cdot$15 \\ 
2-comp & 68$\cdot$85 & 24$\cdot$75 & 68$\cdot$75 & 10$\cdot$66 & 14$\cdot$07 & 19$\cdot$51 & 25$\cdot$35 & 63$\cdot$35 & 72$\cdot$18 \\ 
Gauss & 78$\cdot$40 & 19$\cdot$08 & 75$\cdot$38 & 9$\cdot$60 & 15$\cdot$31 & 20$\cdot$32 & 25$\cdot$67 & 81$\cdot$27 & 97$\cdot$87 \\ 
\\
   $y$ & 2$\cdot$27 & 3$\cdot$74 & 6 & 7$\cdot$99 & 9$\cdot$66 & 11$\cdot$61 & 14$\cdot$93 & 20$\cdot$17 & 22 \\ 
nonpar & 1082$\cdot$44 & 164$\cdot$38 & 10$\cdot$39 & 20$\cdot$06 & 13$\cdot$86 & 6$\cdot$90 & 8$\cdot$18 & 33$\cdot$98 & 48$\cdot$76 \\ 
2-comp & 1098$\cdot$02 & 173$\cdot$48 & 8$\cdot$87 & 22$\cdot$28 & 14$\cdot$67 & 6$\cdot$38 & 8$\cdot$05 & 37$\cdot$87 & 50$\cdot$42 \\ 
Gauss & 1242$\cdot$02 & 203$\cdot$51 & 4$\cdot$89 & 24$\cdot$31 & 18$\cdot$97 & 6$\cdot$86 & 3$\cdot$51 & 34$\cdot$35 & 51$\cdot$69 \\ 
\end{tabular}
\caption{Relative errors ($\times100$) of the three estimators compared to the true densities at selected values for $y$ averaged over 10000 series of length 750. `Gauss' stands for Gaussian state-dependent distributions, `2-comp' for two component Gaussian mixtures and `nonpar' for nonparametric Gaussian mixtures. }
\label{tab:relerr750}
\end{table}
\begin{table}[b]
\centering
\begin{tabular}{rrrrrrrrrr}
$y$& $-$15$\cdot$45 & $-$13$\cdot$77 & $-$11$\cdot$22 & $-$9$\cdot$05 & $-$7$\cdot$26 & $-$5$\cdot$3 & $-$2$\cdot$86 & $-$0$\cdot$21 & 1$\cdot$56 \\ 
nonpar & 93$\cdot$48 & 22$\cdot$08 & 5$\cdot$20 & 13$\cdot$91 & 23$\cdot$87 & 4$\cdot$15 & 44$\cdot$15 & 43$\cdot$74 & 42$\cdot$45 \\ 
2-comp & 107$\cdot$55 & 24$\cdot$15 & 4$\cdot$55 & 13$\cdot$14 & 25$\cdot$37 & 3$\cdot$70 & 45$\cdot$70 & 45$\cdot$71 & 52$\cdot$13 \\ 
Gauss & 131$\cdot$48 & 29$\cdot$27 & 4$\cdot$63 & 10$\cdot$96 & 24$\cdot$73 & 4$\cdot$26 & 43$\cdot$53 & 54$\cdot$12 & 36$\cdot$54 \\ 
\\
$y$    & $-$9$\cdot$36 & $-$6$\cdot$36 & $-$2$\cdot$71 & $-$0$\cdot$68 & 0$\cdot$5 & 1$\cdot$67 & 3$\cdot$71 & 7$\cdot$36 & 10$\cdot$36 \\ 
nonpar & 64$\cdot$26 & 18$\cdot$72 & 65$\cdot$32 & 9$\cdot$64 & 12$\cdot$73 & 19$\cdot$32 & 25$\cdot$16 & 55$\cdot$36 & 57$\cdot$27 \\ 
2-comp & 68$\cdot$83 & 18$\cdot$89 & 68$\cdot$83 & 10$\cdot$82 & 13$\cdot$69 & 19$\cdot$64 & 25$\cdot$35 & 57$\cdot$85 & 57$\cdot$48 \\ 
Gauss & 81$\cdot$72 & 11$\cdot$59 & 74$\cdot$11 & 9$\cdot$84 & 14$\cdot$53 & 19$\cdot$51 & 25$\cdot$26 & 82$\cdot$88 & 98$\cdot$48 \\ 
\\
$y$ & 2$\cdot$27 & 3$\cdot$74 & 6 & 7$\cdot$99 & 9$\cdot$66 & 11$\cdot$61 & 14$\cdot$93 & 20$\cdot$17 & 22 \\ 
nonpar & 1105$\cdot$51 & 171$\cdot$73 & 9$\cdot$21 & 20$\cdot$85 & 13$\cdot$87 & 5$\cdot$64 & 5$\cdot$01 & 33$\cdot$39 & 48$\cdot$50 \\ 
2-comp & 1110$\cdot$86 & 179$\cdot$99 & 7$\cdot$31 & 22$\cdot$93 & 15$\cdot$77 & 5$\cdot$26 & 4$\cdot$55 & 38$\cdot$11 & 51$\cdot$00 \\ 
Gauss & 1226$\cdot$72 & 202$\cdot$15 & 4$\cdot$57 & 23$\cdot$86 & 18$\cdot$44 & 6$\cdot$39 & 2$\cdot$53 & 36$\cdot$33 & 54$\cdot$61 \\ 
\end{tabular}
\caption{Relative errors ($\times100$) of the three estimators compared to the true densities at selected values for $y$ averaged over 10000 series of length 2000. `Gauss' stands for Gaussian state-dependent distributions, `2-comp' for two component Gaussian mixtures and `nonpar' for nonparametric Gaussian mixtures. }
\label{tab:relerr2000}
\end{table}
\begin{table}[t]
\centering
\begin{tabular}{rrrrrrrrrr}
$y$ & $-$15$\cdot$45 & $-$13$\cdot$77 & $-$11$\cdot$22 & $-$9$\cdot$05 & $-$7$\cdot$26 & $-$5$\cdot$3 & $-$2$\cdot$86 & $-$0$\cdot$21 & 1$\cdot$56 \\ 
nonpar & 71$\cdot$64 & 15$\cdot$52 & 3$\cdot$84 & 15$\cdot$11 & 22$\cdot$65 & 3$\cdot$32 & 44$\cdot$27 & 38$\cdot$79 & 39$\cdot$63 \\ 
 2-comp &97$\cdot$12 & 20$\cdot$84 & 2$\cdot$74 & 14$\cdot$30 & 25$\cdot$50 & 1$\cdot$97 & 46$\cdot$85 & 41$\cdot$65 & 53$\cdot$99 \\ 
 Gauss &131$\cdot$16 & 29$\cdot$32 & 3$\cdot$44 & 11$\cdot$02 & 24$\cdot$72 & 3$\cdot$92 & 43$\cdot$63 & 54$\cdot$49 & 36$\cdot$31 \\ 
 \\
    & $-$9$\cdot$36 & $-$6$\cdot$36 & $-$2$\cdot$71 & $-$0$\cdot$68 & 0$\cdot$5 & 1$\cdot$67 & 3$\cdot$71 & 7$\cdot$36 & 10$\cdot$36 \\ 
nonpar & 64$\cdot$72 & 12$\cdot$76 & 67$\cdot$33 & 9$\cdot$86 & 13$\cdot$44 & 20$\cdot$00 & 24$\cdot$80 & 51$\cdot$26 & 49$\cdot$53 \\ 
 2-comp &68$\cdot$63 & 13$\cdot$22 & 70$\cdot$28 & 10$\cdot$72 & 14$\cdot$22 & 20$\cdot$43 & 25$\cdot$70 & 54$\cdot$01 & 52$\cdot$37 \\ 
 Gauss &82$\cdot$78 & 7$\cdot$25 & 74$\cdot$02 & 10$\cdot$10 & 14$\cdot$28 & 19$\cdot$34 & 25$\cdot$53 & 83$\cdot$61 & 98$\cdot$67 \\ 
\\
    & 2$\cdot$27 & 3$\cdot$74 & 6 & 7$\cdot$99 & 9$\cdot$66 & 11$\cdot$61 & 14$\cdot$93 & 20$\cdot$17 & 22 \\ 
nonpar & 1111$\cdot$82 & 174$\cdot$24 & 8$\cdot$69 & 21$\cdot$15 & 13$\cdot$98 & 5$\cdot$27 & 3$\cdot$54 & 35$\cdot$19 & 51$\cdot$30 \\ 
 2-comp &1120$\cdot$81 & 183$\cdot$07 & 6$\cdot$64 & 22$\cdot$93 & 16$\cdot$05 & 4$\cdot$88 & 2$\cdot$99 & 38$\cdot$98 & 52$\cdot$27 \\ 
 Gauss &1222$\cdot$89 & 202$\cdot$18 & 4$\cdot$28 & 23$\cdot$57 & 18$\cdot$16 & 6$\cdot$26 & 1$\cdot$83 & 38$\cdot$10 & 56$\cdot$45 \\ 
\end{tabular}
\caption{Relative errors ($\times100$) of the three estimators compared to the true densities at selected values for $y$ averaged over 10000 series of length 5000. `Gauss' stands for Gaussian state-dependent distributions, `2-comp' for two component Gaussian mixtures and `nonpar' for nonparametric Gaussian mixtures. }
\label{tab:relerr5000}
\end{table}
\clearpage
\section{An additional simulation scenario with linearly dependent state-dependent distributions}
We consider an additional simulation scenario in which the state-dependent distributions are linearly dependent and, moreover, they differ not in location but rather in scale. 
Let $g_{\beta(a,b)}(x)= \{\Gamma(a+b)\}/\{\Gamma(a)\Gamma(b)\}x^{a-1}(1-x)^{b-1}\mathds{1}_{(0,1)}(x)$ denote the density of the Beta distribution, and let $g_{\beta(a,b)}(x; l, s) = g_{\beta(a,b)}\{(x-l)/s\}/s$ denote the Beta density translated by $l$ and scaled by $s$. 
Again we construct a three-state hidden Markov model with transition probability matrix 
\begin{equation*}
\Gamma=
\begin{pmatrix}
\text{0$\cdot$5} & \text{0$\cdot$25} & \text{0$\cdot$25}\\
\text{0$\cdot$4} & \text{0$\cdot$4} & \text{0$\cdot$2}\\
\text{0$\cdot$2} & \text{0$\cdot$2} & \text{0$\cdot$6}
\end{pmatrix}
\end{equation*}
and let the location parameter $\mu$ for the state-dependent mixtures of the first and the second state follow a Beta distribution $g_{\beta(2,11)}(\mu;-3,20)$, while the scale parameter $\sigma$ is uniformly distributed on the interval $(0$$\cdot$$9,1$$\cdot$$5)$ in the first state and uniformly distributed on the interval $(4,6)$ in the second state. The density in the third state is a linear combination of those of the first and second states: $f_{3,0}(y)=0$$\cdot$$4 f_{1,0}(y)+ 0$$\cdot$$6 f_{2,0}(y)$. 
\begin{figure}[h]
\centering
\includegraphics[width=0.4\textwidth]{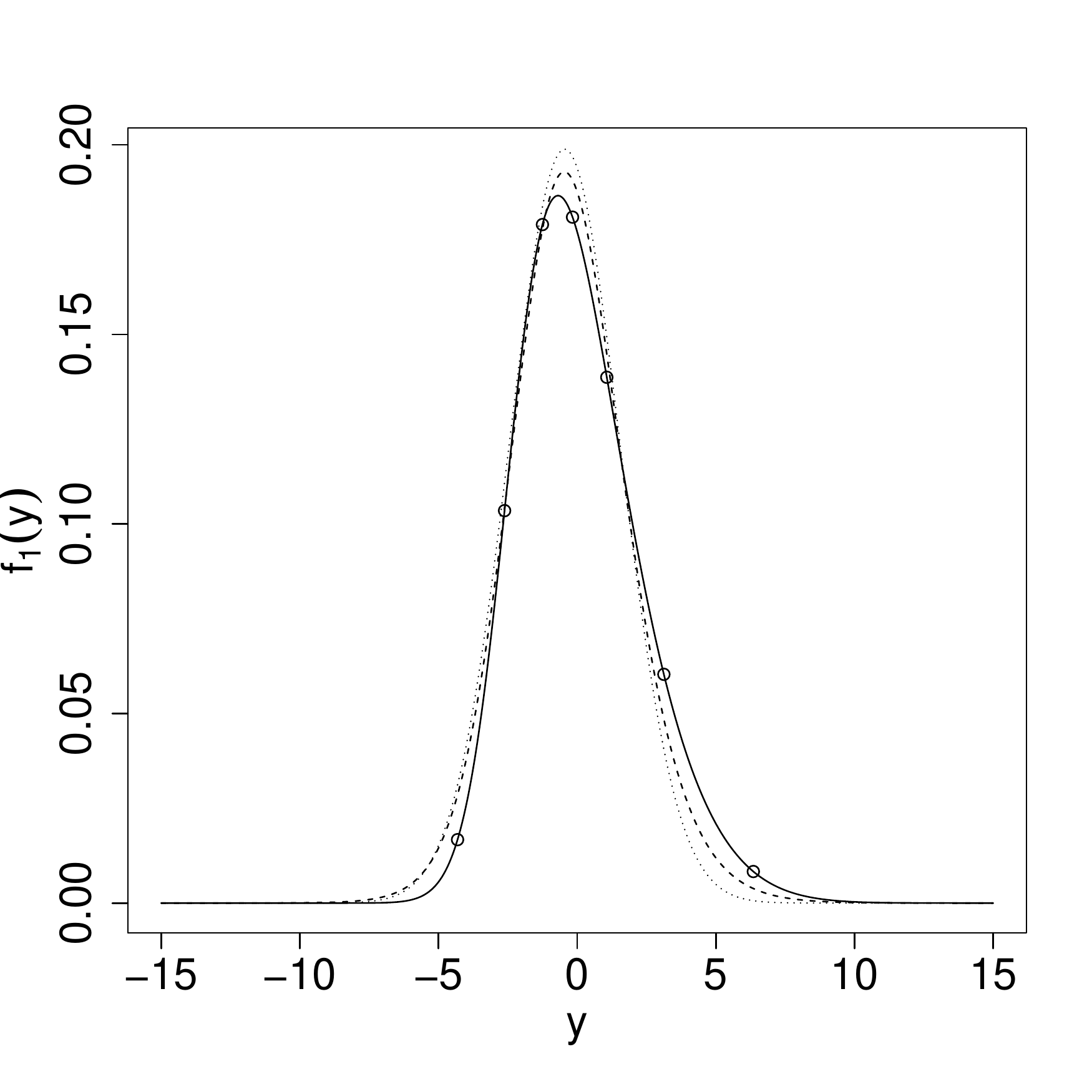}
\includegraphics[width=0.4\textwidth]{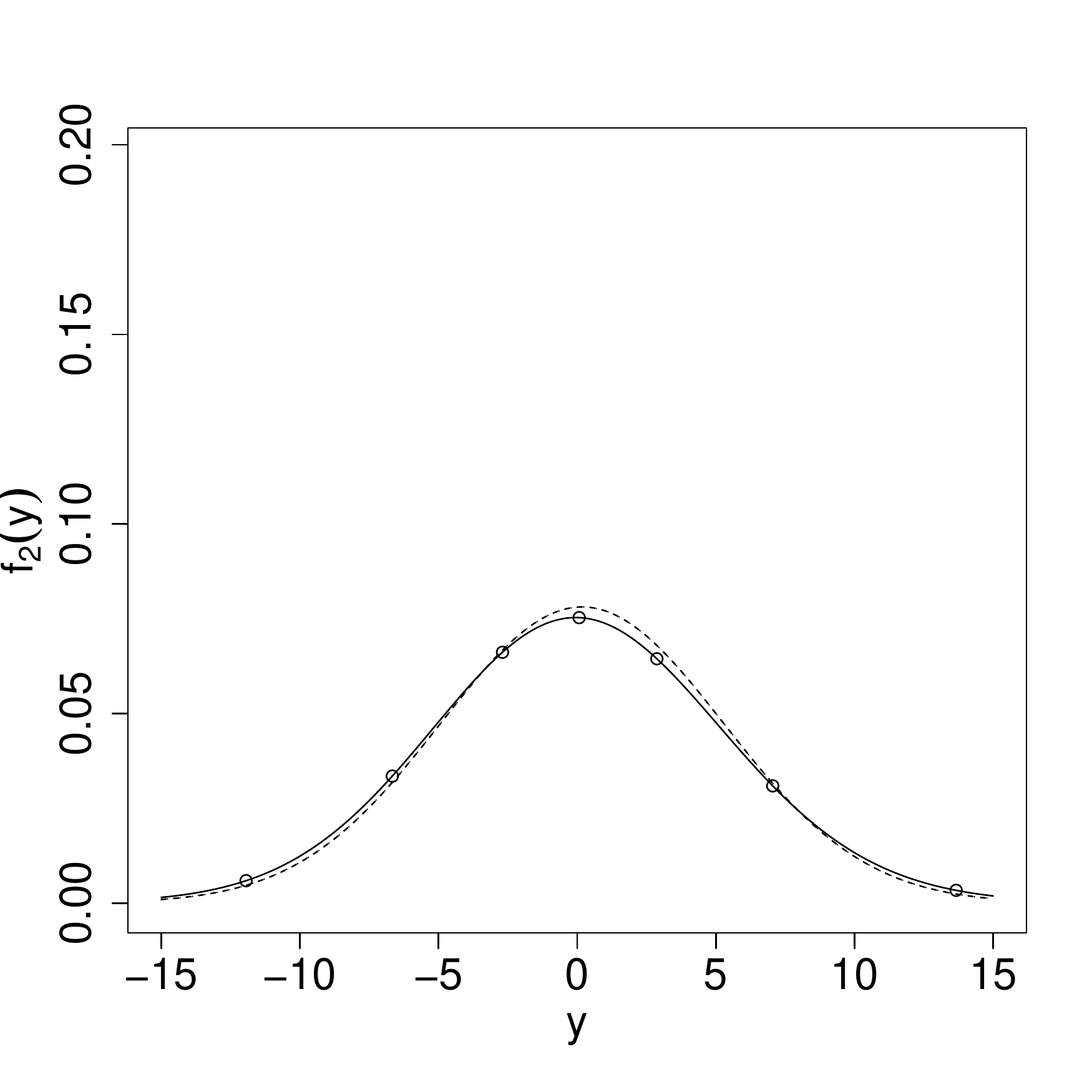}
\includegraphics[width=0.4\textwidth]{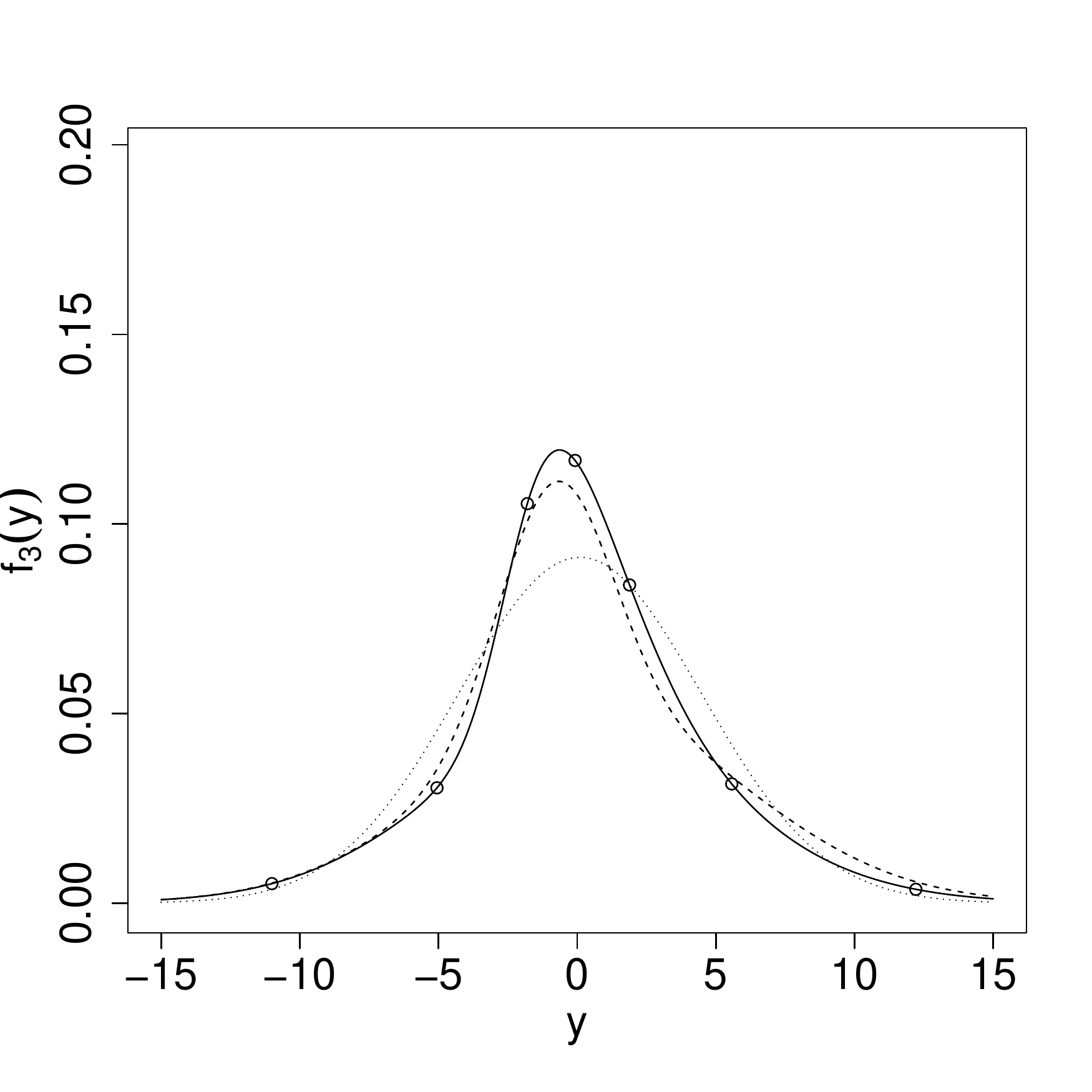}
\includegraphics[width=0.4\textwidth]{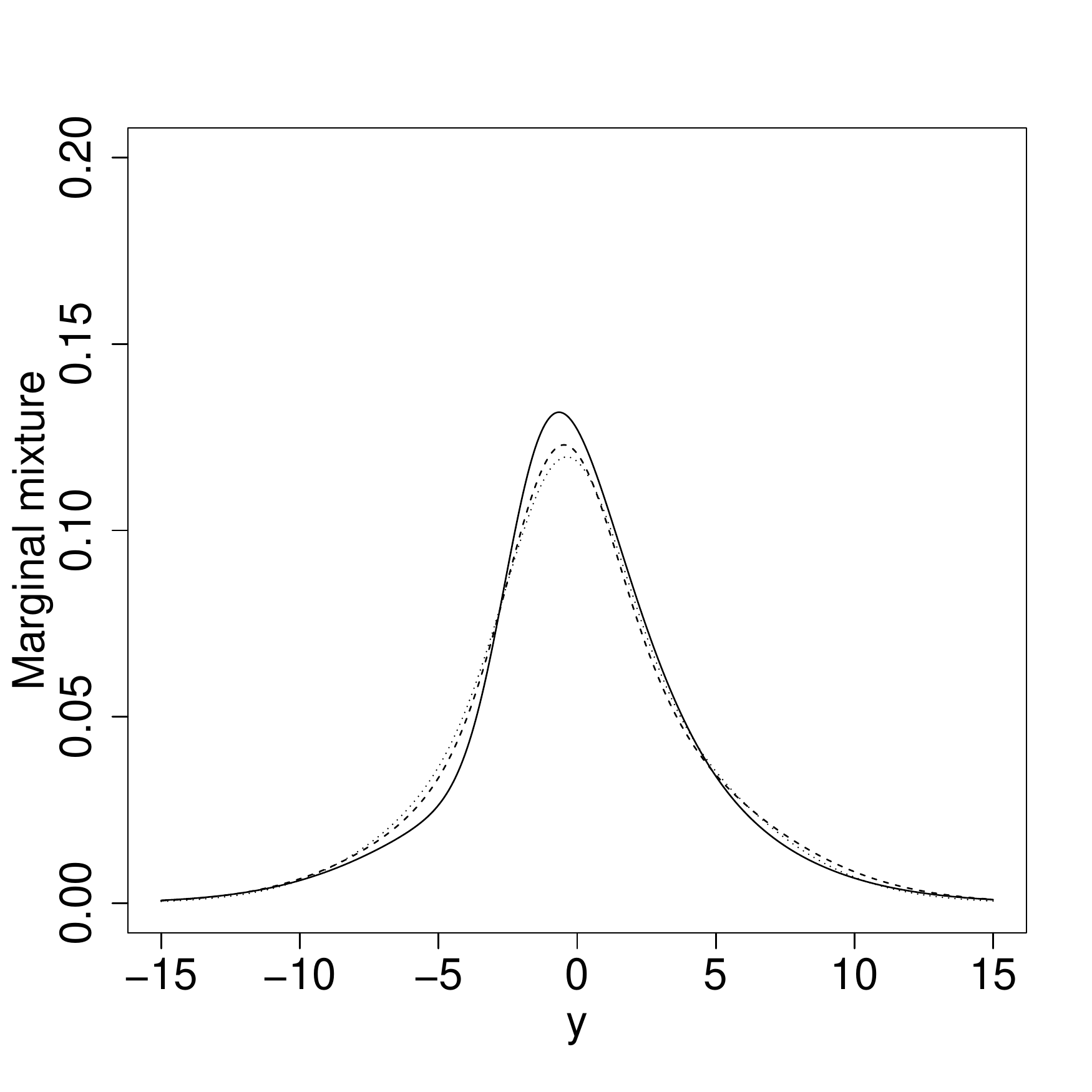}
\caption{State dependent densities and marginal density of the hidden Markov model. Solid line: true densities, dashed line: nonparametric maximum likelihood estimator, dotted line: Gaussian maximum likelihood estimate. }
\label{fig:estimates}
\end{figure}
We only compare the nonparametric and a parametric Gaussian maximum likelihood estimator. 
In Fig.~\ref{fig:estimates} we plot estimates of the state-dependent densities and the marginal mixture of the hidden Markov model when using the stationary distributions of the estimated transition probability matrices and the estimated state-dependent densities. In the first state where the density is slightly skew, we observe that the nonparametric maximum likelihood estimator performs better than the parametric estimator, whereas in the second state, where due to the large scale parameters the density is close to being symmetric, the estimators yield similar results. In the third state, the advantage of the nonparametric estimator is obvious, especially in tracing the left tail and the peak of the density. 
\begin{table}[ht]
\centering
\begin{tabular}{rrrrrrrr}
1st state \\
$y$ & $-$4$\cdot$31 & $-$2$\cdot$62 & $-$1$\cdot$25 & $-$0$\cdot$17 & 1$\cdot$07 & 3$\cdot$12 & 6$\cdot$35 \\ 
nonparametric & 22$\cdot$87 & 12$\cdot$85 & 7$\cdot$88 & 18$\cdot$34 & 15$\cdot$61 & 27$\cdot$22 & 72$\cdot$45 \\ 
parametric & 27$\cdot$84 & 10$\cdot$48 & 6$\cdot$69 & 19$\cdot$28 & 20$\cdot$16 & 32$\cdot$96 & 94$\cdot$34 \\ 
\\
 2nd state\\
$y$   & $-$11$\cdot$94 & $-$6$\cdot$67 & $-$2$\cdot$69 & 0$\cdot$07 & 2$\cdot$87 & 7$\cdot$05 & 13$\cdot$66 \\ 
nonparametric & 20$\cdot$61 & 9$\cdot$25 & 4$\cdot$77 & 8$\cdot$12 & 6$\cdot$59 & 7$\cdot$00 & 40$\cdot$76 \\ 
parametric & 21$\cdot$43 & 4$\cdot$92 & 2$\cdot$78 & 4$\cdot$78 & 5$\cdot$38 & 3$\cdot$98 & 33$\cdot$89 \\ 
\\
 3rd state\\
$y$ & $-$11$\cdot$01 & $-$5$\cdot$06 & $-$1$\cdot$8 & $-$0$\cdot$08 & 1$\cdot$89 & 5$\cdot$57 & 12$\cdot$21 \\ 
nonparametric & 22$\cdot$97 & 37$\cdot$08 & 15$\cdot$74 & 15$\cdot$93 & 5$\cdot$40 & 22$\cdot$23 & 41$\cdot$73 \\ 
parametric & 31$\cdot$77 & 49$\cdot$92 & 20$\cdot$36 & 21$\cdot$09 & 2$\cdot$20 & 30$\cdot$80 & 49$\cdot$30 \\ 
\end{tabular}
\caption{Relative errors ($\times100$) of the two estimators compared to the true densities at selected values for $y$ averaged over 10000 replications.  `Gauss' stands for Gaussian state-dependent distributions, `2-comp' for two component Gaussian mixtures and `nonpar' for nonparametric Gaussian mixtures.}
\label{tab:esterror}
\end{table}

For the points plotted in Fig.~\ref{fig:estimates} we evaluate the relative errors of the estimators and provide the results averaged over 10000 replications in Table \ref{tab:esterror}. We observe that for the first state, except for two points, the nonparametric estimator yields better results than the parametric estimator. In the second state the parametric estimator yields somewhat better results when estimating the nearly symmetric density. For the third state the nonparametric estimator yields substantially better results than the parametric estimator. 

The absolute errors for the estimates of the transition probabilities are reported in Table \ref{tab:tpm2}. The nonparametric estimator yields slightly better results in the second and third state. 
\begin{table}[h]
\begin{tabular}{cccc}
 & $K^{-1}\sum_{k=1}^K |\hat{\alpha}_{j,k}-\alpha_{j,k}|$ &  $K^{-1}\sum_{k=1}^K |\tilde{\alpha}_{j,k}-\alpha_{j,k}|$ \\ 
State $j=1$ &  11$\cdot$93 & 11$\cdot$71\\
State $j=2$ &9$\cdot$65   & 9$\cdot$93  \\
State $j=3$ & 4$\cdot$52 & 5$\cdot$34
\end{tabular} 
\caption{Absolute errors of estimated transition probabilities ($\times$100) averaged over 10000 simulations. Nonparametric estimator ($\hat{\alpha}_{j,k}$) and parametric estimator ($\tilde{\alpha}_{j,k}$) $(j,k=1,\hdots,K)$. }
\label{tab:tpm2}
\end{table}


%

\bibliographystyle{ims}
\bibliography{db}